\documentclass[11pt]{article}
\usepackage{amsfonts}
\usepackage{latexsym}
\usepackage{amsmath}
\usepackage{amssymb}
\usepackage{amsthm}
\usepackage{amsmath}

\usepackage{mathtools}
\usepackage{mathrsfs}

\usepackage{verbatim}
\usepackage{hyperref}
\usepackage{bibspacing}
 \usepackage{tikz}
\usepackage{pgfplots}
\usepgfplotslibrary{fillbetween}

\newtheorem{thm}{Theorem}[section]
\newtheorem*{thm*}{Theorem}
\newtheorem{cor}[thm]{Corollary}

\newtheorem{lem}[thm]{Lemma}
\newtheorem{lemma}[thm]{Lemma}
\newtheorem{prop}[thm]{Proposition}
\newtheorem*{prop*}{Proposition}
\newtheorem{proposition}[thm]{Proposition}

\newtheorem*{conj*}{Conjecture}
\newtheorem{defn}[thm]{Definition}
\newtheorem*{dfn*}{Definition}
\theoremstyle{definition}
\newtheorem{rem}[thm]{\textbf{Remark}}
\newtheorem*{rmk*}{Remark}
\newtheorem*{fact*}{Fact}

\theoremstyle{proof}

\newcommand{\dsa}[2]{\left . \frac{d}{ds} #1 \right |_{#2^+}}
\newcommand{\dta}[2]{\left . \frac{d}{dt} #1 \right |_{#2^+}}
\newcommand{\dt}[1]{\left . \frac{d}{dt} #1 \right |_{0^+}}
\newcommand{\dtl}[1]{\left . \underline{\frac{d}{dt}} #1 \right |_{0^+}}
\newcommand{\dtll}[1]{\underline{\frac{d}{dt}} #1 |_{0^+}}
\newcommand{\dht}[1]{\left . \frac{d#1}{dt}\right |_{0^+}}

\DeclareMathOperator{\interior}{\textnormal{int}}
\DeclareMathOperator{\closure}{\textnormal{cl}}
\DeclareMathOperator{\conv}{\textnormal{conv}}
\DeclareMathOperator{\sspan}{\textnormal{span}}
\DeclareMathOperator{\supp}{\textnormal{supp}}
\DeclareMathOperator{\Rad}{\mathcal{R}}
\newcommand{\Lip}{\mathcal{L}}

\newcommand{\GL}{\textrm{GL}}

\newcommand{\norm}[1]{\left\Vert#1\right\Vert}

\newcommand{\abs}[1]{\left\vert#1\right\vert}
\newcommand{\set}[1]{\left\{#1\right\}}
\newcommand{\brac}[1]{\left(#1\right)}
\newcommand{\scalar}[1]{\left \langle #1 \right \rangle}

\newcommand{\di}[1]{\left<#1\right>}
\newcommand{\Real}{\mathbb{R}}
\newcommand{\R}{\mathbb{R}}

\newcommand{\Y}{Y}

\renewcommand{\S}{\mathbb{S}}
\newcommand{\F}{\mathcal{F}}
\renewcommand{\H}{\mathcal{H}}
\newcommand{\I}{\mathcal{I}}
\newcommand{\J}{L}
\newcommand{\y}{\mathbf{y}}

\newcommand{\eps}{\epsilon}
\renewcommand{\L}{\mathscr{L}}  \newcommand{\U}{U}\renewcommand{\J}{\mathcal{J}}

\renewcommand{\SS}{\mathcal{S}}
\newcommand{\Id}{\textrm{Id}}
\newcommand{\usc}{\textrm{USC}}
\newcommand{\h}{\ell}

\numberwithin{equation}{section}

\numberwithin{equation}{section}

\begin{document}

\renewcommand*{\thefootnote}{\fnsymbol{footnote}}

\author{Emanuel Milman\textsuperscript{$*$,$\dagger$}
\and
Shahar Shabelman\textsuperscript{$*$,$\diamondsuit$}
\and
Amir Yehudayoff\textsuperscript{$\|$,$*$,$\ddagger$}\\\\
\emph{dedicated to the memory of Paolo Gronchi}
}

\footnotetext{$^*$ Department of Mathematics, Technion-Israel Institute of Technology, Haifa 32000, Israel.}
\footnotetext{$^{\|}$ Department of Computer Science, University of Copenhagen, Denmark.} 
\footnotetext{$^\dagger$ Email: emilman@tx.technion.ac.il.}
\footnotetext{$\!\!^\diamondsuit$ Email: shabelman@campus.technion.ac.il.}
\footnotetext{$^\ddagger$ Email: amir.yehudayoff@gmail.com. }

\begingroup    \renewcommand{\thefootnote}{}    
    \footnotetext{The research leading to these results is part of a project that has received funding from the European Research Council (ERC) under the European Union's Horizon 2020 research and innovation programme (grant agreement No 101001677). A.Y. is supported by a DNRF Chair grant.}
\endgroup

\title{Fixed and Periodic Points of the Intersection Body Operator}

\date{\nonumber} 
\maketitle

\begin{abstract}
The intersection body $IK$ of a star body $K$ in $\R^n$ was introduced by E.~Lutwak following the work of H.~Busemann, and plays a central role in the dual Brunn-Minkowski theory. We show that when $n \geq 3$, $I^2 K = c K$ iff $K$ is a centered ellipsoid, and hence $I K = c K$ iff $K$ is a centered Euclidean ball, answering long-standing questions by Lutwak, Gardner, and Fish--Nazarov--Ryabogin--Zvavitch. An equivalent formulation of the latter in terms of non-linear harmonic analysis states that a non-negative $\rho \in L^\infty(\S^{n-1})$ satisfies $\Rad(\rho^{n-1}) = c \rho$ for some $c > 0$ iff $\rho$ is 
constant, where $\Rad$ denotes the spherical Radon transform.  Our proof is entirely geometrical: we recast the iterated intersection body equation as an Euler-Lagrange equation for a certain volume functional under radial perturbations, derive new formulas for the volume of $I K$, and introduce a continuous version of Steiner symmetrization for Lipschitz star bodies, which (surprisingly) yields a useful radial perturbation exactly when $n\geq 3$. 
\end{abstract}

\section{Introduction}

A Borel set $K$ in $\R^n$ is called star-shaped if
\[
K = \{ r \theta : r \in [0,\rho_K(\theta)] ~,~ \theta \in \S^{n-1} \} ,
\]
for some (Borel) function $\rho_K : \S^{n-1} \rightarrow \R_+$ called its radial function, where $\S^{n-1}$ denotes the Euclidean unit sphere in $\R^n$. The set $K$ is called a star body if $\rho_K$ is positive and continuous (and hence $K$ is necessarily compact); the family of star bodies in $\R^n$ is denoted by $\SS_n$. A compact convex set with non-empty interior is called a convex body.

The intersection body $I K$ of a star body $K$ in $\R^n$ was introduced and studied by Lutwak in \cite{Lutwak-intersection-bodies}, who defined $I K$ as the star body given by
\begin{equation} \label{eq:intro-IB}
\rho_{IK}(\theta) = |K \cap \theta^{\perp}|_{n-1} .
\end{equation}
Here and throughout, we use $\abs{L}_k$ to denote the $k$-dimensional Hausdorff measure of $L$, and often omit the subscript $k$ when it is equal to the dimension of $L$'s affine hull. Remarkably, it was shown by Busemann \cite{Busemann-IntersectionBodyIsConvex} (see also \cite[Theorem 8.1.10]{GardnerGeometricTomography2ndEd}) that when $K$ is an origin-symmetric convex body then $I K$ is itself convex. 
Busemann also showed \cite{Busemann-IntersectionBodyInq} (see also \cite[Corollary 9.4.5]{GardnerGeometricTomography2ndEd} or \cite[Section 10.10]{Schneider-Book-2ndEd}) that if $K$ is convex then
\begin{equation} \label{eq:intro-II}
\abs{IK} \leq \abs{I B_K} , \end{equation}
where $B_K$ is a centered Euclidean ball having the same volume as $K$, with equality when $n \geq 3$ if and only if $K$ is a centered ellipsoid. 

Lutwak's definition of the intersection body (\ref{eq:intro-IB}) and Busemann's intersection inequality (\ref{eq:intro-II}) may be extended to arbitrary bounded Borel sets $K \subset \R^n$ (even though the star-shaped $IK$ may not be a star body in general), and the characterization of equality in (\ref{eq:intro-II}) when $n \geq 3$ remains valid for general star bodies $K$ (see Petty's work \cite{Petty-CentroidSurfaces}), and up to null-sets, for bounded Borel sets~$K$ as well (see Corollary 7.5 in Gardner's work~\cite{Gardner-DualAffineQuermassintegrals}).
 Note that the case $n=2$ is excluded since $I K = 2 U K$ for any origin-symmetric star body $K$ in $\R^2$, where $U$ denotes a 90-degree rotation, and so $\abs{I K} = 4 \abs{K} = \abs{I B_K}$. Intersection bodies play an essential role in the dual Brunn-Minkowski theory and in Geometric Tomography, in particular in relation to the solution of the Busemann-Petty problem --- we refer to \cite{Gardner-BP-3dim, GKS, Koldobsky-Book, Lutwak-intersection-bodies,Zhang-Correction-4dim}, \cite[Chapter 8]{GardnerGeometricTomography2ndEd} and the references therein for additional information.

\medskip
 
Let  $I : \SS_n \rightarrow \SS_n$ denote the intersection body operator. Our main results in this work are the following characterizations.

\begin{thm} \label{thm:intro-main}
Let $K$ be a star body in $\R^n$, $n \geq 3$. Then $I^2 K = c K$ for some $c > 0$ iff $K$ is a centered ellipsoid. 
\end{thm}
This provides a positive answer to questions of Lutwak \cite[Open Problem 12.8]{Lutwak-Selected} and Gardner \cite[Open Problem 8.6, Case $i=n-1$]{GardnerGeometricTomography2ndEd}, who asked whether centered ellipsoids are indeed the only star bodies for which $I^2 K = c K$. As a consequence, we easily deduce
the following corollary. 
\begin{cor} \label{cor:intro-main}
Let $K$ be a star body in $\R^n$, $n \geq 3$. Then $I K = c K$ for some $c > 0$ iff $K$ is a centered Euclidean ball. 
\end{cor}
This provides a complete answer to questions of Gardner \cite[Open Problem 8.7, Case $i=n-1$]{GardnerGeometricTomography2ndEd} and Fish--Nazarov--Ryabogin--Zvavitch \cite[Question]{FNRZ-IntersectionBodyOperator}, who asked what are the fixed points of the intersection body operator $I : \SS_n \rightarrow \SS_n$ when $n \geq 3$. The authors of \cite{FNRZ-IntersectionBodyOperator} also asked what the periodic points of $I$ are, and Theorem \ref{thm:intro-main} provides a partial answer in this direction. Note that in \cite[Open Problems 8.6-8.7]{GardnerGeometricTomography2ndEd}, an even more general family of operators depending on a parameter $i \in \{1,\ldots,n-1\}$ is considered (see \cite[Corollary 9.8]{Grinberg-Zhang} for a solution to the case $i=1$). 
We emphasize that both Theorem \ref{thm:intro-main} and Corollary \ref{cor:intro-main} are new already for the class of convex bodies $K$ (in which case the more technical parts of our proof may be simplified, but the heart of our argument remains novel). 

\begin{rem}
Naturally, both statements above are false for $n=2$. Indeed, $I^2 K = 4 K$ for any origin-symmetric star body $K$ in $\R^2$, and $I K = 2 K$ holds for any $K$ invariant under $U$. Consequently, any attempt at a proof must crucially use the assumption that $n \geq 3$. Interestingly, our proof will use the fact that $B_\infty^n$ is \emph{not} a subset of $2 B_1^n$ when $n \geq 3$, where $B_p^n$ denotes the unit ball of $\L_p^n$. \end{rem}

\begin{rem} \label{rem:intro-gen}
As we will see in the proof, Theorem \ref{thm:intro-main} and Corollary \ref{cor:intro-main} actually hold under the more general assumption that $K$ is a star-shaped bounded Borel set in $\R^n$ ($n \geq 3$) satisfying $I^2 K = c K$ or $I K = c K$ up to null-sets, in which case it is possible to modify $K$ on a null-set so that either $\rho_K \equiv 0$ (when $|K| = 0$) or else $K$ is a centered ellipsoid or Euclidean ball, respectively.
\end{rem}

The above results may be equivalently formulated as results in \emph{non-linear} harmonic analysis. Let $\Rad : C(\S^{n-1}) \rightarrow C(\S^{n-1})$ denote the spherical Radon (or Funk) transform, defined by $\Rad(f)(u) := \int_{\S^{n-1} \cap u^{\perp}} f(\theta) d\sigma_{\S^{n-1} \cap u^{\perp}}(\theta)$,
where $\sigma_{\S}$ denotes the Haar probability measure on the sphere $\S$. It is easy to see that $\Rad$ is a bounded operator in $L^p(\S^{n-1})$ ($p \in [1,\infty]$), and so its action continuously extends to $L^2(\S^{n-1})$. 
Passing to polar coordinates, we see that $\rho_{IK}(u) = \omega_{n-1} \Rad(\rho_K^{n-1})$ for an appropriate $\omega_{n-1} > 0$, and so in view of Remark \ref{rem:intro-gen}, Corollary \ref{cor:intro-main} translates to the following:
\begin{cor} \label{cor:intro-Radon}
Let $\rho$ denote a non-negative function in $L^\infty(\S^{n-1})$, $n \geq 3$. 
Then, as functions in $L^\infty(\S^{n-1})$,
\[
\Rad(\rho^{n-1}) = c \; \rho \; \text{ for some $c > 0$} \;\; \text{ iff } \;\; \text{$\rho$ is constant.}
\]
\end{cor}
Alternatively, let $\xi^{\wedge}$ denote the Fourier transform of a distribution $\xi$ in $\R^n$, and recall that if $\xi$ is an even homogeneous distribution of degree $-n+\alpha$ then $\xi^{\wedge}$ is even and homogeneous of degree $-\alpha$ (see \cite[Lemma 2.21 and Theorem 3.8]{Koldobsky-Book} for more information). Applying Corollary \ref{cor:intro-Radon} to $\rho = 1 / \xi_0$, we have:
\begin{cor} \label{cor:intro-Fourier}
Let $\xi$ denote a $1$-homogeneous extension to $\R^n$ of an even measurable function $\xi_0 \geq \delta > 0$ on $\S^{n-1}$, $n \geq 3$. Then, as distributions,
\[
(\xi^{-n+1})^{\wedge} \cdot \xi \equiv c \text{ on $\R^n \setminus \{0\}$ for some $c > 0$} \;\; \text{ iff } \;\; \text{$\xi$ is a multiple of the Euclidean norm.}
\]
\end{cor}
These non-linear results do not seem to be amenable to non-perturbative harmonic analytic methods. However, when $\norm{\rho-1}_{L^\infty} \leq \eps_n$ for some small enough $\eps_n > 0$ depending solely on $n$, Corollary \ref{cor:intro-Radon} was established using perturbative Fourier methods by Fish--Nazarov--Ryabogin--Zvavitch \cite{FNRZ-IntersectionBodyOperator}. Moreover, these authors showed that when $n \geq 3$ and $K$ is a star body sufficiently close to a centered Euclidean ball $B_n$ in the Banach-Mazur distance, then $I^m K \rightarrow B_n$ as $m \rightarrow \infty$ in the Banach-Mazur distance, thereby deducing for such $K$ that if $I^m K = c K$ for some $m \geq 1$ then necessarily $K$ is a centered ellipsoid. 
\smallskip

We proceed to provide a sketch of the argument leading up to Theorem \ref{thm:intro-main}, and describe several new ingredients which we believe are of independent interest. 

\subsection{Variational approach via continuous Steiner symmetrization}

The significance of the equation 
\begin{equation} \label{eq:intro-I2K}
I^2 K = c K
\end{equation}
stems from the fact that it is the Euler-Lagrange equation for the functional
\begin{equation} \label{eq:intro-F}
\F_c(K) := |I K| - (n-1) c |K| ,
\end{equation}
characterizing its stationary points under radial perturbations. 
A precise statement is somewhat technical (see Proposition \ref{prop:stationary}), since one has to carefully specify an appropriate class of \emph{admissible} radial perturbations $\{K_t\}$ of $K_0 = K$. When $K$ is a star body with Lipschitz continuous radial function $\rho_K$ (``Lipschitz star body"), it turns out that \emph{continuous} Steiner symmetrization $\{S_u^t K\}_{t \in [0,1]}$ in a.e.~direction $u \in \S^{n-1}$ provides such an admissible perturbation. This idea is quite general, and may be used to generate and solve various additional geometric equations --- we present some examples in the concluding Section \ref{sec:conclude}. 

\smallskip

Continuous Steiner symmetrization for graphical domains has its origins in the work of P\'olya--Szeg\"o \cite[Note B]{PolyaSzego-Book}.
For the class of convex bodies, continuous Steiner symmetrization is a particular case of a shadow system \cite{RogersShephard-ShadowSystems,Shephard-ShadowSystems}, a well-established and extremely useful tool, which has been successfully used to resolve numerous geometric extremization problems (see e.g. \cite{CampiGronchi-LpBusemannPettyCentroid,   CampiGronchi-VolumeProductInqs, CFPP-EasyBusemannAndCampiGronchi,  MeyerReisner-SantaloViaShadowSystems,EMilmanYehudayoff-AffineQuermassintegrals, RogersShephard-ShadowSystems,  Saroglou-GeneralizedBCD, Shephard-ShadowSystems} to name just a few). Continuous Steiner symmetrization has been used by Rogers \cite{Rogers-BLL}, Brascamp--Lieb--Luttinger \cite{BrascampLiebLuttinger} and Christ \cite{Christ-KPlaneTransform} to treat general compact sets and measurable functions, by first approximating them with simpler objects and then applying a fiberwise gradual symmetrization; it thus underlies many symmetrization results (see e.g.~\cite{PaourisPivovarov-Survey} and the references therein). This type of fiberwise continuous symmetrization was subsequently extended to directly apply to general measurable sets in $\R^n$ by Brock \cite{Brock-ContSteiner,Brock-ContSteiner2}, who defined it up to null-sets, studied its properties and derived various applications. See \cite{BGGK-PolyaSzego} for an account, unified treatment and extension of numerous types of symmetrizations which have appeared in the literature. 

\smallskip

However, there is little literature on the \emph{geometric} properties of Steiner symmetrization of star bodies, for which one expects to have some control over their corresponding boundaries along the symmetrization process. While the classical Steiner symmetrization $S_u K$ of Lipschitz star bodies $K$ has been recently studied in \cite{LinXi-LipschitzStarBodySymmetrization}, we are not aware of any prior works involving a continuous version $\{S_u^t K\}_{t \in [0,1]}$ in the class of star bodies, which is what we require for our variational approach. In particular, we introduce the first explicit definition of $\{S_u^t K\}_{t \in [0,1]}$ for a.e.~$u \in \S^{n-1}$ in the class of Lipschitz star bodies $K$, and establish the {\it a priori} non-obvious fact that $\{S_u^t K\}$ remain (uniformly Lipschitz) star bodies for all $t \in [0,1]$, giving rise to a genuine \emph{radial} and a.e.~\emph{differentiable} perturbation of $K$. Even the seemingly trivial statement that $|S_u^t K|\equiv |K|$ remains constant requires a careful verification. See Sections \ref{sec:Steiner} through \ref{sec:admissible} for precise definitions and details. 

\smallskip

Using some fairly standard results in Harmonic Analysis and Sobolev spaces (see Appendix \ref{sec:appendix}), one can show that any solution to (\ref{eq:intro-I2K}) when $n \geq 3$ must have $C^\infty$ smooth (and in particular Lipschitz) radial function $\rho_K$. Thus, after quite a bit of technical preparation, our starting point for characterizing those (Lipschitz) star bodies $K$ satisfying (\ref{eq:intro-I2K}) is  that
\begin{equation} \label{eq:intro-dIdt}
\left . \frac{d}{dt} \right |_{t=0^+} |I (S_u^t K)|  = 0 \;\;\;  \text{ for a.e. } u \in \S^{n-1} . 
\end{equation}

\subsection{New formulas for $|I K|$}

Our next ingredient, which appears to be novel even for convex bodies $K$, is a new formula for the volume of the intersection body $| I K|$. We first extend the domain of $I$ to include non-negative, bounded and compactly supported Borel measurable functions $g$ on $\R^n$, by defining $I(g)$ to be the star-shaped set 
\[
\rho_{I(g)}(u) = \int_{u^{\perp}} g(y) \, dy ,
\]
noting that $I(K) = I(1_K)$ for any compact $K$. 
Now assuming that the following limit exists, define
\begin{equation} \label{eq:intro-I0}
\I_0(g) := \lim_{p \rightarrow -1^+} \; \frac{p+1}{n} \int_{\R^n} \cdots \int_{\R^n} \Delta(x_1,\ldots,x_n)^{p} g(x_1) \cdots g(x_n) \, dx_1 \ldots dx_n ,
\end{equation}
where $\Delta(x_1,\ldots,x_m)$ denotes the $m$-dimensional volume of the parallelotope linearly spanned by $x_1,\ldots,x_m \in \R^n$. 
Finally, if $K$ is a compact set, denote $\I_0(K) = \I_0(1_K)$ assuming that the limit exists. In this work we show that $| I K | = \I_0(K)$ for any star body $K$. More generally, we introduce the following condition.

\begin{defn}[Radially negligible boundary] \label{def:intro-radially-negligible}
A compact set $K \subset \R^n$ is said to have radially negligible boundary if
\[
\int_{\S^{n-1}} |\partial K \cap \R_+ u|_1 \;  d\sigma_{\S^{n-1}}(u) =  0 .
\]
\end{defn}
\begin{rem} \label{rem:intro-radially-negligible}
A general star-shaped compact set may not have radially negligible boundary, but a star body $K$ does, since $\text{int}(K) \cap \R_+ u =  [0,\rho_K(u)) u$ and hence $\partial K \cap \R_+ u = \{\rho_K(u) u\}$.  
\end{rem}

\begin{thm} \label{thm:intro-I0K}
Let $f \in C_c(\R^n,\R_+)$, and let $K$ be a compact set in $\R^n$ with radially negligible boundary. Then, the limit in (\ref{eq:intro-I0}) exists for $g \in \{ f , 1_K \}$ and
\[
|I(f)| = \I_0(f) \text{ and } |I K| = \I_0(K) .
\]
\end{thm}

Here $C_c(\R^n,\R_+)$ denotes the family of non-negative continuous functions on $\R^n$ with compact support. 
Using the Steiner concavity of the integrand in (\ref{eq:intro-I0}) when $p < 0$ (see Subsection \ref{subsec:Steiner} for details), we immediately deduce the following corollary for $f \in C_c(\R^n,\R_+)$; the case of general compact sets $K$ is obtained by approximation.  We denote by $S_u f$ the Steiner symmetrization of $f$ via a layer-cake representation. 
\begin{cor} \label{cor:intro-SuK}
Let $f \in C_c(\R^n,\R_+)$, and let $K$ be a compact set in $\R^n$. Then, for all $u \in \S^{n-1}$,
\begin{equation} \label{eq:intro-IfISuf}
|I(f)| \leq |I(S_u f)| \text{ and } |I K| \leq |I (S_u K)| . 
\end{equation}
\end{cor}

Applying an appropriate sequence of symmetrizations, it is known \cite[Lemma 9.4.3]{BuragoZalgallerBook} that $K_m = S_{u_m}\ldots S_{u_1} K$ converges to $B_K$ in the Hausdorff metric, and thanks to the continuity of $|I K|$ under Hausdorff convergence,
the classical Busemann intersection inequality (\ref{eq:intro-II}) for general compact sets immediately follows. Surprisingly, the inequality $|I(f)| \leq |I(S_u f)|$ for a \emph{single} application of Steiner symmetrization ($f \in C_c(\R^n,\R_+)$) has only recently been established by Adamczak--Paouris--Pivovarov--Simanjuntak \cite{APPS-LpCentroidBodies}. More generally, the results of \cite{APPS-LpCentroidBodies} apply to dual $L^p$-centroid bodies for $p \geq 0$ and when $n/p$ is an integer to $p \in [-1,0)$ as well (extending the intersection body case of $p=-1$). The proof in \cite{APPS-LpCentroidBodies} is fairly intricate, and requires several limiting arguments, thereby precluding (as far as we can see) any attempt to study the cases of equality in (\ref{eq:intro-IfISuf}), which are crucial for our variational approach. Our definition of $\I_0(K)$ in (\ref{eq:intro-I0}) also involves a limit, leading to a similar difficulty, but fortunately, for a nice class of compact sets $K$, we are able to calculate this limit as follows. 

\begin{defn}[$u$-finite compact set] \label{def:intro-u-finite}
Let $u \in \S^{n-1}$. 
A compact set $K$ in $\R^n$ is called $u$-finite if for a.e.~$y \in u^{\perp}$, $K \cap (y + \R u)$ consists of a finite disjoint union of closed intervals (each of positive length). 
\end{defn}

\begin{thm} \label{thm:intro-Iu}
Let $K$ be a $u$-finite compact set in $\R^n$. Then, the limit in (\ref{eq:intro-I0}) for $g = 1_K$ exists and $\I_0(K) = \I_u(K)$, where
\[
\I_u(K) := \frac{2}{n} \int_{(P_{u^\perp} K)^n} \Delta(\tilde y_1,\ldots,\tilde y_{n-1})^{-1} |R_\y \cap \theta_\y^{\perp}|_{n-1} \, dy_1 \ldots dy_n . 
\]
Here $P_{u^{\perp}} K$ denotes the orthogonal projection of $K$ onto $u^{\perp}$, $\y = (y_1,\ldots,y_n) \in (u^\perp)^n$, 
$R_\y = \{ (s^1,\ldots,s^n)  \in \R^n :  y_i + s^i u \in K \, , \, i=1,\ldots,n\}$, $\theta_\y$ denotes the element of $\S^{n-1}$ satisfying $\sum_{i=1}^n \theta_\y^i y_i = 0$ (this linear dependency is unique up to sign for a.e.~$\y$), and $(\tilde y_1,\ldots,\tilde y_{n-1})$ denote the $n-1$ rows of the $(n-1) \times n$ matrix whose $n$ columns in $u^{\perp}$ are $(y_1,\ldots,y_n)$. 
\end{thm}

In other words, $\I_u(K)$ is actually independent of $u$ and coincides with $\I_0(K)$ for any $u \in \S^{n-1}$ for which $K$ is $u$-finite. 
By results of Lin and Xi \cite{LinXi-LipschitzStarBodySymmetrization}, for a.e.~$u\in \S^{n-1}$, a given Lipschitz star body $K$ is not only $u$-finite, but in fact satisfies a stronger property we call $u$-multi-graphicality (see Definition \ref{def:mgs} and Theorem \ref{thm:Lip-multi-graphical}), which allows us to introduce a well-defined notion of continuous Steiner symmetrization $\{S_u^t K\}_{t \in [0,1]}$ in the direction of $u$. Thanks to Theorem~\ref{thm:intro-Iu}, we are able to analyze the behaviour of $\I_0(S_u^t K)$ using the formula for $\I_u(S_u^t K)$.

\subsection{Equality analysis}

Having Theorem \ref{thm:intro-Iu} at hand, we obtain the following characterization.   

\begin{thm} \label{thm:intro-dIdt}
Let $K$ be a Lipschitz star body in $\R^n$. Then, there exists $\U \subseteq \S^{n-1}$ of full-measure such that for all $u \in \U$, $[0,1] \ni t \mapsto |I (S_u^t K)| = \I_0(S_u^t K) = \I_u(S_u^t K)$ is non-decreasing, and the following derivative exists and satisfies
\begin{equation} \label{eq:intro-positive-derivative}
\left . \frac{d}{dt} \right |_{t=0^+} |I (S_u^t K)|  \geq 0 .
\end{equation}
 If equality occurs in (\ref{eq:intro-positive-derivative}) for a given $u \in \U$, then for a.e.~$\y = (y_1,\ldots,y_n) \in (P_{u^{\perp}} K)^n$, $R_\y$ consists of a finite disjoint union of rectangles $\{R_\y^k\}$ in $\R^n$, such that for each rectangle $R_\y^k$, either
\begin{enumerate} 
\item \label{it:intro-cond1} $\theta_\y^{\perp}$ essentially does not intersect $R_\y^k$, i.e., $|R_\y^k \cap \theta_\y^{\perp}|_{n-2} = 0$, or  
\item \label{it:intro-cond2} $\theta_\y^{\perp}$ passes through the center $c(R_\y^k)$ of $R_\y^k$, i.e., $\scalar{\theta_\y , c(R_\y^k)} = 0$, or 
\item \label{it:intro-cond3} $\theta_\y^{\perp}$ intersects exactly $n-1$ pairs of opposing facets of the centered rectangle $R_\y^k - c(R_\y^k)$ (and no other facets).
\end{enumerate}
\end{thm}

Perhaps surprisingly, the proof of Theorem \ref{thm:intro-dIdt} is based on Brunn's concavity principle and the characterization of equality in the Brunn-Minkowski inequality for convex bodies, even though $K$ is only assumed to be a Lipschitz star body. 
It is not hard to show that for such $K$'s, there exists a $\delta > 0$ so for all $u \in \S^{n-1}$ and $y \in B_{u^\perp}(\delta)$, $J_u^y := K \cap (y + \R u)$ is an interval (containing $y$ in its interior). Here $B_{E}(p,\delta)$ denotes the Euclidean ball  in $E$ of radius $\delta$ centered at $p$, $B_E(\delta) := B_E(0,\delta)$, and we abbreviate $B_n = B_{\R^n}$. Consequently, for all 
 $\y = (y_1,\ldots,y_n) \in B_{u^\perp}(\delta)^n$, $R_\y$ is a single rectangle $\Pi_{i=1}^n [c^i_{\y} - \h^i_{\y} , c^i_{\y} + \h^i_{\y}]$ containing the origin in its interior, and hence its intersection with $\theta_\y^{\perp}$ violates condition (\ref{it:intro-cond1}). 
As for condition (\ref{it:intro-cond3}), it has the following interesting geometric consequence (see Lemma \ref{lem:all-facets}):
\begin{equation} \label{eq:intro-geometric}
P_{\sspan \theta_\y} B^n_\infty(\h_\y) \subset \interior (2 P_{\sspan \theta_\y} B^n_1(\h_\y)) ,
\end{equation}
where $B^n_\infty(\h) = \Pi_{i=1}^n [-\h^i,\h^i]$ and $B^n_1(\h) = \conv \{\pm \h^i e_i \}_{i=1,\ldots,n}$ are stretched unit balls of $\L_\infty^n$ and $\L_1^n$, respectively. However, when $n \geq 3$, we can always find an open set $\Theta \subset \R^n$ of $\theta$'s (independent of any other parameter) so that (\ref{eq:intro-geometric}) cannot hold, because the inclusions $B_1^n \subset B_\infty^n \subset n B_1^n$ are best possible. Consequently, we deduce that if $u \in \U$ and $\left . \frac{d}{dt} \right|_{t=0^+} |I (S_u^t K)| = 0$, then for a.e.~$\y=(y_1,\ldots,y_n) \in B_{u^{\perp}}(\delta)^n$ which satisfy a linear dependency contained in $\Theta$, condition (\ref{it:intro-cond2}) must hold. We then observe the following lemma. 

\begin{lemma} \label{lem:intro-linear}
Let $f : B \rightarrow \R$ be a function on a centered open Euclidean ball $B \subset \R^{n-1}$, $n \geq 3$, and let $\Theta \subset \R^n$ be a non-empty open set. Assume that for all $\theta \in \Theta$,
for every affinely independent $y_1,\ldots,y_n \in B$ such that $\sum_{i=1}^n \theta_i y_i = 0$, it holds that $\sum_{i=1}^n \theta_i f(y_i) = 0$. 
Then, $f$ must be a linear function on $B \setminus \{0\}$.   
\end{lemma}
Applying this to $f(y)$, the $u$-height of the center of the interval $J_u^y$, we deduce that all mid-points of $J_u^y$ for $y \in  B_{u^\perp}(\delta)$ lie on a common hyperplane through the origin. It remains to adapt to our setting the following criterion of Soltan \cite[Corollary 1]{Soltan-EllipsoidViaMidpoints}, which is a local form of the classical Bertrand--Brunn characterization of ellipsoids.

\begin{thm}[Soltan] \label{thm:intro-Soltan}
Let $K$ be a convex body in $\R^n$. Assume that there is a $p \in \interior K$ and a $\delta > 0$ such that for every direction $u \in \S^{n-1}$, the mid-points of all segments of $K$ parallel to $u$ and passing through $B_n(p,\delta)$ all lie on a common hyperplane. Then, $K$ must be an ellipsoid. 
\end{thm}
\noindent It turns out that it is enough that $K$ is a Lipschitz star body, and for the assumption to hold only for a dense 
set of $u$'s (see Theorem \ref{thm:Soltan} for a precise statement). 

\medskip

We thus conclude that a Lipschitz star body $K$ satisfying (\ref{eq:intro-I2K}) must satisfy (\ref{eq:intro-dIdt}), and hence the equality conditions of Theorem \ref{thm:intro-dIdt} for a.e.~$u \in \S^{n-1}$.
When $n \geq 3$, we deduce by Lemma \ref{lem:intro-linear} that for a.e.~$u \in \S^{n-1}$, all of the mid-points of segments parallel to $u$ passing through $B_n(\delta)$ lie on a common hyperplane through the origin. It now follows that $K$ is a centered ellipsoid by an appropriate version of Theorem \ref{thm:intro-Soltan}. 

\smallskip
Along the way, we also prove the following counterpart to Corollary \ref{cor:intro-SuK} (which in itself does not help in establishing Theorem \ref{thm:intro-main}).

\begin{cor} \label{cor:intro-SuK-equality}
Let $K$ be a Lipschitz star body in $\R^n$, $n \geq 3$. Then, the following statements are equivalent:
\begin{enumerate}
\item $|I K| = |I(S_u K)|$ for all $u \in \S^{n-1}$. 
\item $|I K| = |I(S_u K)|$ for a.e.~$u \in \S^{n-1}$. 
\item $K$ is a centered ellipsoid. 
\end{enumerate}
\end{cor}

\subsection{Organization}

The rest of this work is organized as follows. In Section \ref{sec:prelim} we introduce some standard preliminaries and notation. In Section \ref{sec:formulas} we derive the new formulas for $|I(K)|$ given by $\I_0(K)$ and $\I_u(K)$ and establish Theorems \ref{thm:intro-I0K} and \ref{thm:intro-Iu}. In Section \ref{sec:Steiner} we recall the definition of the classical Steiner symmetrization (establishing Corollary \ref{cor:intro-SuK} along the way), introduce a continuous version for $u$-multi-graphical compact sets, and study its properties. In Section~\ref{sec:Lip} we study the graphical properties of Lipschitz star bodies $K$ and their continuous Steiner symmetrization $\{S_u^t K\}$. In Section \ref{sec:admissible} we show that $\{S_u^t K\}$ constitute an admissible radial perturbation of $K$, and that the equation $I^2 K = c K$ characterizes stationary points for the functional $\F_c$ from (\ref{eq:intro-F}) under such perturbations. In Section \ref{sec:equality} we give several implications of  $\left . \frac{d}{dt} \right|_{t=0^+} |I(S_u^t K)| = 0$ and in particular establish Theorem \ref{thm:intro-dIdt} and Lemma \ref{lem:intro-linear}. In Section~\ref{sec:proof} we conclude the proofs of Theorem \ref{thm:intro-main} (taking into account Remark \ref{rem:intro-gen}) and Corollaries \ref{cor:intro-main} and \ref{cor:intro-SuK-equality}. In Section \ref{sec:conclude} we provide some concluding remarks regarding additional applications of our method. In Appendix \ref{sec:appendix} we show that a solution to $I^2 K = c K$ is necessarily smooth when $n \geq 3$.

\medskip

\noindent
\textbf{Acknowledgments.} We thank Gabriele Bianchi and Richard Gardner for their comments and for informing us of Brock's work. 
We also thank the anonymous referees for carefully reading the manuscript and for their very helpful comments.

\section{Preliminaries and notation} \label{sec:prelim}

We assume that $n \geq 2$ throughout this work. Given a Euclidean space $E$, we denote by $B_E(p,r)$ the closed Euclidean unit ball of radius $r > 0$ in $E$ centered at $p \in E$, abbreviating $B_E(r) = B_E(0,r)$ and $B_E = B_E(1)$. When $E = \R^n$, we simply use $B_n$ instead of $B_E$, and denote by $\S^{n-1} = \partial B_n$ the Euclidean unit sphere. A Euclidean ball, and more generally, an ellipsoid, are called centered if their center is at the origin. 
We denote the unit ball of the normed space $\L_p^n$ by $B_p^n$, and the corresponding norm by $\norm{\cdot}_p$. We denote the $k$-dimensional Hausdorff measure by $\H^k$, sometimes utilizing $|\cdot|_k$ instead. Given a set $A$ in a topological space $X$, we denote by $\interior A$ and $\closure A$ its interior and closure, respectively. The family of continuous functions on $X$ is denoted by $C(X)$. 
The support $\supp(f)$ of a function $f$ on $X$ is defined as $\closure \{ f \neq 0 \}$. The family of compactly supported continuous functions on $X$ is denoted by $C_c(X)$, and the subset of compactly supported non-negative functions is denoted by $C_c(X,\R_+)$. 

\medskip

A Borel subset $K \subset \R^n$ is called star-shaped (with respect to the origin) if
\[
K = \{ r \theta : r \in [0,\rho_K(\theta)] ~,~ \theta \in \S^{n-1} \} ,
\]
for some (Borel) function $\rho_K : \S^{n-1} \rightarrow \R_+$ called its radial function. Note that our definition does not require $K$ to be compact or closed like some authors, to ensure that the intersection body $I(g)$ is star-shaped for a general (say, compactly-supported and bounded) Borel measurable function $g : \R^n \rightarrow \R_+$. When $K$ is a compact set, clearly it is star-shaped iff for any $x \in K$, the interval $[0,x]$ from the origin to $x$ is contained in $K$, iff $\lambda K \subseteq K$ for all $\lambda \in [0,1]$. 

When $\rho_K$ is positive and continuous (and hence $K$ is necessarily compact), $K$ is called a star body (with respect to the origin). When $K$ is a star body, note that $\interior(K) = \{ r \theta : r \in [0,\rho_K(\theta)) , \theta \in \S^{n-1} \}$ (see e.g. \cite[Lemma 3.2]{Zhu-OrliczCentroidInqForStarBodies}), so that every ray $\R_+ \theta$ intersects $\partial K$ at exactly one point and $\partial K = \{ \rho_K(\theta) \theta : \theta \in \S^{n-1} \}$. We will at times consider $\rho_K$ as a function on $\R^n \setminus \{0\}$ by extending it as a $-1$-homogeneous function, so that $\rho_K(x) = 1$ iff $x \in \partial K$. 
The corresponding gauge function is defined as $\norm{x}_K := \inf \{ t \geq 0 : x \in t K\}$ --- it is a $1$-homogeneous function on $\R^n$ satisfying $\norm{x}_K = 1$ iff $x \in \partial K$ and thus coincides with $1/\rho_K(x)$ (the norm notation is standard, despite not satisfying the triangle inequality nor being an even function in general). 

More generally, we will say that $K$ is star-shaped (star body) with respect to $p \in \R^n$ if $K-p$ is star-shaped (star body), and that $K$ is star-shaped (star body) with respect to a subset $P \subset \R^n$ if it is star-shaped (star body) with respect to all $p \in P$.

\medskip

Here and throughout, we use $dy$ to denote integration with respect to the Haar volume measure on the corresponding homogeneous space where $y$ is defined. Integrating in polar coordinates, we have for any star-shaped set $K$ in $\R^n$,
\begin{equation} \label{eq:vol-polar}
|K| = \frac{1}{n} \int_{\S^{n-1}} \rho^n_{K}(\theta) \;  d\theta = \omega_n \int_{\S^{n-1}} \rho^n_{K}(\theta) \, d\sigma_{\S^{n-1}}(\theta) , 
\end{equation}
where $\sigma_{\S}$ denotes the Haar \emph{probability} measure on the sphere $\S$ and we set $\sigma_n := |\S^{n-1}|_{n-1}$
 and $\omega_n := |B_n| = \sigma_n / n$.

\smallskip
We will make use of the following particular case of the Blaschke--Petkantschin formula (see \cite[Theorem 7.2.1]{SchneiderWeil-Book}), stating that for any non-negative Borel measurable function $g$ on $(\R^n)^{n-1}$,
\begin{align}
\nonumber & \int_{(\R^n)^{n-1}} g(x_1,\ldots,x_{n-1}) \, dx_1 \ldots dx_{n-1} \\
\nonumber & = \int_{G_{n,n-1}} \int_{E^{n-1}} \Delta(x_1,\ldots,x_{n-1}) g(x_1,\ldots,x_{n-1}) \, dx_1 \ldots dx_{n-1} dE \\
\label{eq:BP} & = \frac{1}{2} \int_{\S^{n-1}} \int_{(\theta^{\perp})^{n-1}} \Delta(x_1,\ldots,x_{n-1}) g(x_1,\ldots,x_{n-1}) \, dx_1 \ldots dx_{n-1} d\theta ,
\end{align}
where $G_{n,n-1}$ denotes the Grassmannian of all $(n-1)$-dimensional linear subspaces of $\R^n$, equipped with its natural Haar volume measure 
normalized so that $\int_{G_{n,n-1}} dE = \frac{1}{2} |\S^{n-1}|_{n-1}$.
Here $\Delta(x_1,\ldots,x_m)$ denotes the $m$-dimensional volume of the parallelotope linearly spanned by $x_1,\ldots,x_m \in \R^n$.

\smallskip
We will use the standard fact (see e.g.~\cite[Lemma 1.3.3]{Groemer}) that \begin{equation} \label{eq:double-integration}
\sigma_{\S^{n-1}} = \int_{\S^{n-1}} \sigma_{\S^{n-1} \cap \theta^{\perp}} \, d\sigma_{\S^{n-1}}(\theta) .
\end{equation}
The spherical Radon (or Funk) transform $\Rad : C(\S^{n-1}) \rightarrow C(\S^{n-1})$ is defined as
\[
\Rad(f)(u) := \int_{\S^{n-1} \cap u^\perp} f(\theta) \, d\sigma_{\S^{n-1} \cap u^\perp}(\theta) . 
\]
It follows immediately by Jensen's inequality and (\ref{eq:double-integration}) that $\Rad$ is a contraction in $L^2(\S^{n-1})$ (in fact, any $L^p(\S^{n-1})$),
and so by density its action extends to this entire space. The resulting operator $\Rad: L^2(\S^{n-1}) \rightarrow L^2(\S^{n-1})$ is symmetric, namely,
\begin{equation} \label{eq:Rad-symmetric}
\int_{\S^{n-1}} \Rad(f) g \, du = \int_{\S^{n-1}} f \Rad(g) \, du \;\;\; \forall f,g \in L^2(\S^{n-1}) . 
\end{equation}

\medskip

The Minkowski sum of two sets $A,B \subset \R^n$ is defined as $A + B  = \{ a+b : a\in A , b \in B\}$. By the Brunn-Minkowski inequality \cite{GardnerSurveyInBAMS, GardnerGeometricTomography2ndEd, Gruber-ConvexAndDiscreteGeometry,  Schneider-Book-2ndEd}, if $K,L$ are convex bodies in $\R^n$ then $|K+L|^{1/n} \geq |K|^{1/n} + |L|^{1/n}$. An equivalent form is given by Brunn's concavity principle \cite[Theorem 8.4]{Gruber-ConvexAndDiscreteGeometry}, stating that if $K$ is a convex body in $\R^n$ and $\theta \in \S^{n-1}$ then
\begin{equation} \label{eq:Brunn}
\R \ni t \mapsto g(t) = |K \cap (t \theta + \theta^{\perp})|^{\frac{1}{n-1}}_{n-1} \text{ is concave on its support.}
\end{equation}
If $g$ is constant on $[a,b] \subseteq \supp g$, then by the equality cases of the Brunn-Minkowski inequality \cite[Theorem 7.1.1]{Schneider-Book-2ndEd}, $K \cap  (t \theta + \theta^{\perp})$ must be translates of each other for all $t \in [a,b]$. 

\medskip

Given $u \in \S^{n-1}$, we denote $L_u = \sspan(u)$, and given $y \in u^{\perp}$, we let $L_u^y = y + L_u$ be the line through $y$ in the direction of $u$. 
We denote by $P_E$ the orthogonal projection in $\R^n$ onto a linear subspace $E$. 

\medskip
Lastly, given a function $f : J \rightarrow \R$ on an interval $J$ so that $f(t)$ is differentiable from the right at $t=a$, we denote by $\dta{f(t)}{a} := \left  .\frac{d}{dt} \right |_{t=a^+} f(t)$ its right-derivative at $t=a$.

\section{New formulas for $|I K|$} \label{sec:formulas}

\subsection{Radially negligible boundary}

\begin{lemma} \label{lem:radially-negligible}
For any Borel set $L$ in $\R^n$,
\begin{equation} \label{eq:radially-negligible}
\int_{\S^{n-1}} |L \cap \R_+ u|_{1} \, d\sigma_{\S^{n-1}}(u)  = 0  \;\; \text{ iff } \;\; 
\int_{\S^{n-1}} |L \cap \theta^{\perp}|_{n-1} \, d\sigma_{\S^{n-1}}(\theta)  = 0 
.
\end{equation}
In particular, $K$ has radially negligible boundary according to Definition  \ref{def:intro-radially-negligible} iff
\[
\int_{\S^{n-1}} |\partial K \cap \theta^{\perp}|_{n-1} \, d\sigma_{\S^{n-1}}(\theta)  = 0  .
\]
\end{lemma}
\begin{proof}
Integrating in polar coordinates on each $\S^{n-1} \cap \theta^{\perp}$, we have
\begin{align*}
& \int_{\S^{n-1}} |L \cap \theta^{\perp}|_{n-1} \, d\sigma_{\S^{n-1}}(\theta) \\
& = \int_{\S^{n-1}} \int_{\theta^{\perp}} 1_{L}(y) \;  dy \, d\sigma_{\S^{n-1}}(\theta) \\
& = \omega_{n-1} \int_{\S^{n-1}} \int_{\S^{n-1} \cap \theta^{\perp}} \int_0^\infty r^{n-2} 1_{L}(r u) \, dr \, d\sigma_{\S^{n-1} \cap \theta^{\perp}}(u)  \, d\sigma_{\S^{n-1}}(\theta) \\
& = \omega_{n-1} \int_{\S^{n-1}} \int_0^\infty r^{n-2} 1_{L}(r u) \, dr \, d\sigma_{\S^{n-1}}(u) ,
\end{align*}
where we used (\ref{eq:double-integration}). 
Since the measures given by $r^{n-2} dr$ and $dr$ on $\R_+$ are mutually absolutely continuous, the assertion follows.
\end{proof} 

As explained in Remark \ref{rem:intro-radially-negligible}, the first variant in (\ref{eq:radially-negligible}) immediately verifies that any star body $K$ has radially negligible boundary, but we shall employ the second variant in the sequel. 

\subsection{First formula --- $\I_0(K)$}

Let $f$ be a non-negative, bounded and compactly supported Borel measurable function on $\R^n$. Recall that $I(f)$ denotes the star-shaped set in $\R^n$ whose radial function is given by
\[
\rho_{I(f)}(\theta) = \int_{\theta^{\perp}} f(y) \;  dy. 
\]
Let
\begin{equation} \label{eq:I0}
\overline{\I}_0(f) := \limsup_{p \rightarrow -1^+} \; \frac{p+1}{n} \int_{\R^n} \cdots \int_{\R^n} \Delta(x_1,\ldots,x_n)^{p} f(x_1) \cdots f(x_n) \, dx_1 \ldots dx_n .
\end{equation}
Similarly, define $\underline{\I}_0(f)$ by replacing the $\limsup$ with a $\liminf$, and if the two limits agree, define $\I_0(f)$ to be their common value. Here and throughout we use $p \rightarrow a^+$ to denote taking the limit to $a$ from the right. We denote $\I(K) = \I(1_K)$ for $\I \in \{ \overline{\I}_0, \underline{\I}_0, \I_0 \}$ (assuming that the limit exists in the latter case). The following is an extended version of Theorem \ref{thm:intro-I0K}:

\begin{thm} \label{thm:I0}
Let $f$ be a non-negative, bounded and compactly supported Borel measurable function on $\R^n$.
\begin{enumerate}
\item If $f$ is lower semi-continuous then $|I(f)| \leq \underline{\I}_0(f)$.
\item If $f$ is upper semi-continuous then $|I(f)| \geq \overline{\I}_0(f)$. 
\item In particular, if $K \subset \R^n$ is compact, then $|I K| \geq \overline{\I}_0(K)$, and if $f$ is continuous then the limit in (\ref{eq:I0}) exists and $|I(f)| = \I_0(f)$. 
\item If $K \subset \R^n$ is compact with radially negligible boundary then the limit in (\ref{eq:I0}) exists and $|IK| = \I_0(K)$. 
\end{enumerate}
\end{thm}

For the proof, we will make use of the following standard lemma:
\begin{lemma} \label{lem:weak-convergence}
Given $\theta \in \S^{n-1}$ and $R > 0$, the (finite) Borel measures $\mu_p := \frac{p+1}{2} \abs{\scalar{\cdot,\theta}}^p 1_{B_n(R)} \H^n$ converge weakly to $\mu_{-1} := 1_{B_n(R)} \H^{n-1}|_{\theta^{\perp}}$ as $p \rightarrow -1^+$, in the sense that for every bounded continuous function $f$ on $\R^n$ the following limit exists and is equal to
\begin{equation} \label{eq:weak-convergence}
\lim_{p \rightarrow -1^+} \int f \, d\mu_p = \int f \, d\mu_{-1} . 
\end{equation}
Moreover, the following hold:
\begin{enumerate}
\item For any bounded lower semi-continuous function $f_l$, $\liminf_{p \rightarrow -1^+} \int f_l \, d\mu_p \geq \int f_l \, d\mu_{-1}$. 
\item For any bounded upper semi-continuous function $f_u$, $\limsup_{p \rightarrow -1^+} \int f_u \, d\mu_p \leq \int f_u \, d\mu_{-1}$. 
\item For any continuity set $K \subseteq \R^n$ of $\mu_{-1}$, namely a Borel set such that $\mu_{-1}(\partial K) = 0$, the limit in (\ref{eq:weak-convergence}) exists and (\ref{eq:weak-convergence}) holds for $f=1_K$. 
\end{enumerate}
\end{lemma}
\begin{proof}
The convergence (\ref{eq:weak-convergence}) for any bounded continuous $f$ immediately reduces by Fubini's theorem to the corresponding statement in dimension $n=1$, namely that $\frac{p+1}{2} 1_{[-R,R]}(t) |t|^p dt$ converges weakly to the delta-measure at the origin $\delta_0$ as $p \rightarrow -1^+$, which is straightforward to verify. The other assertions follow by the Portmanteau theorem \cite[Theorem 13.16]{KlenkeBook3rdEd} (see also \cite[Corollary 2.2.6]{Bogachev-WeakConvergenceBook}).
\end{proof}

\begin{proof}[Proof of Theorem \ref{thm:I0}]
Let $f_l,f_u$ be bounded non-negative compactly-supported lower and upper semi-continuous functions on $\R^n$, respectively. Let $K$ be a compact set in $\R^n$ with radially negligible boundary. Let $R > 0$ be large enough so that $\supp(f_l),\supp(f_u) , K \subset B_n(R)$. 
Note that by Lemma \ref{lem:radially-negligible}, since $K$ has radially negligible boundary then it is a continuity set for $\mu_{-1}^{\theta} := 1_{B_n(R)} \H^{n-1}|_{\theta^{\perp}}$ for a.e.~$\theta \in \S^{n-1}$.

\smallskip
Consequently, by Lemma \ref{lem:weak-convergence}, for any $x_1,\ldots,x_{n-1} \in \theta^{\perp}$ such that $\Delta(x_1,\ldots,x_{n-1}) > 0$, we have
\begin{align*}
  & \; \limsup_{p \rightarrow -1^+} \frac{p+1}{2} \int_{\R^n} \Delta(x_1,\ldots,x_{n-1},x_n)^p f_u(x_n) \, dx_n \\
 =  & \; \limsup_{p \rightarrow -1^+} \frac{p+1}{2} \int_{\R^n} \Delta(x_1,\ldots,x_{n-1})^p \abs{\scalar{x_n,\theta}}^p f_u(x_n) \, dx_n \\
 \leq & \; \Delta(x_1,\ldots,x_{n-1})^{-1} \int_{\theta^{\perp}} f_u(y) \, dy ,
\end{align*}
with a reversed inequality for the $\liminf$ and $f_l$, and equality of the limit for $1_K$ and a.e.~$\theta \in \S^{n-1}$. As $\Delta(x_1,\ldots,x_{n-1}) > 0$ 
for a.e.~$(x_1,\ldots,x_{n-1}) \in (\theta^{\perp})^{n-1}$, it follows by integration in polar coordinates and Fubini's theorem that
\begin{align}
\nonumber  |I(f_u)| & = \frac{1}{n} \int_{\S^{n-1}} \brac{\int_{\theta^{\perp}} f_u(y) dy }^n d\theta \\
\nonumber  & = \frac{1}{n} \int_{\S^{n-1}} \int_{(\theta^{\perp})^{n-1}} f_u(x_1) \cdots f_u(x_{n-1}) dx_1 \ldots dx_{n-1} \int_{\theta^{\perp}} f_u(y) \, dy  \, d\theta \\
\label{eq:DCT1} & \geq  \frac{1}{n} \int_{\S^{n-1}} \int_{(\theta^{\perp})^{n-1}} f_u(x_1) \cdots f_u(x_{n-1})   \Delta(x_1,\ldots,x_{n-1}) \\
\nonumber & \;\; \brac{\limsup_{p \rightarrow -1^+} \frac{p+1}{2} \int_{\R^n} \Delta(x_1,\ldots,x_n)^p f_u(x_n) dx_n } dx_1 \ldots dx_{n-1} \, d\theta ,
\end{align}
with a reversed inequality for the $\liminf$ and $f_l$, and equality of the limit for $1_K$. Assuming we could exchange integration and limit above 
(while respecting the direction of the inequality), after an application of the Tonelli--Fubini theorem we could continue
\begin{align}
\label{eq:DCT1b}  \geq \frac{1}{n}  \limsup_{p \rightarrow -1^+} \frac{p+1}{2} \int_{\R^n} \int_{\S^{n-1}} \int_{(\theta^{\perp})^{n-1}} &  f_u(x_1) \cdots f_u(x_{n-1})   \Delta(x_1,\ldots,x_{n-1}) \\
\nonumber &  \Delta(x_1,\ldots,x_n)^p \, dx_1 \ldots dx_{n-1} \, d\theta \; f_u(x_n) \, dx_n ,
\end{align}
and so by the Blaschke--Petkantschin formula (\ref{eq:BP}),
the previous quantity is equal to
\begin{align*}
& = \limsup_{p \rightarrow -1^+} \frac{p+1}{n} \int_{\R^n} \int_{(\R^n)^{n-1}} f_u(x_1) \cdots f_u(x_{n-1}) f_u(x_n) \Delta(x_1,\ldots,x_n)^p \, dx_1 \ldots dx_{n-1} \, dx_n ,
\end{align*}
with a reversed inequality in (\ref{eq:DCT1b}) for the $\liminf$ and $f_l$, and equality of the limit for $1_K$, thereby concluding the proof of all assertions. 

It remains to justify the exchange of integration and limit in (\ref{eq:DCT1}). Let $M \in (0,\infty)$ be such that $0 \leq f_l,f_u \leq M$, and recall that $\supp(f_l),\supp(f_u), K \subset B_n(R)$. Then for $f \in \{f_l,f_u,1_K\}$, for all $\theta \in \S^{n-1}$ and $x_1,\ldots,x_{n-1} \in \theta^{\perp}$ such that $\Delta(x_1,\ldots,x_{n-1}) > 0$, and for all $p > -1$, the integrand in (\ref{eq:DCT1}) may be bounded from above as follows:
\begin{align*}
& f(x_1) \cdots f(x_{n-1}) \Delta(x_1,\ldots,x_{n-1}) \frac{p+1}{2} \int_{\R^n} \Delta(x_1,\ldots,x_n)^p f(x_n) \, dx_n \\
& \leq M^n \Pi_{i=1}^{n-1} 1_{B_n(R)}(x_i) \Delta(x_1,\ldots,x_{n-1})  \frac{p+1}{2} \int_{B_n(R)} \Delta(x_1,\ldots,x_{n-1})^p \abs{\scalar{x_n,\theta}}^p  \, dx_n \\
& = M^n \Pi_{i=1}^{n-1} 1_{B_n(R)}(x_i) \Delta(x_1,\ldots,x_{n-1})^{1+p}  \frac{p+1}{2} \int_{B_n(R)}  \abs{\scalar{x_n,\theta}}^p  \, dx_n \\
& \leq M^n R^{(n-1)(1+p)} \Pi_{i=1}^{n-1} 1_{B_n(R)}(x_i) \frac{p+1}{2} \int_{B_n(R)}  \abs{\scalar{x_n,\theta}}^p  \, dx_n  . 
\end{align*}
The function $R^{(n-1)(1+p)} \frac{p+1}{2} \int_{B_n(R)}  \abs{\scalar{x_n,\theta}}^p  dx_n$ is continuous in $p \in (-1,0]$, and converges to $|B_n(R) \cap \theta^{\perp}|_{n-1} = R^{n-1} \omega_{n-1}$ as $p \rightarrow -1^+$. Consequently, there is a constant $C_R$ such that this function is bounded from above by $C_R$ uniformly for all $p \in (-1,0]$, and we conclude that the integrand in (\ref{eq:DCT1}) is bounded from above by 
\[
\leq C_R M^n \Pi_{i=1}^{n-1} 1_{B_n(R)}(x_i) . 
\]
Since
\[
\int_{\S^{n-1}} \int_{(\theta^{\perp})^{n-1}} C_R M^n \Pi_{i=1}^{n-1} 1_{B_n(R)}(x_i) \, dx_1 \ldots dx_{n-1} \, d\theta < \infty, \]
the exchange of integration and limit (while respecting the direction of the inequality)  in (\ref{eq:DCT1}) for $f_\ell$, $f_u$ and $1_K$ is justified by Fatou's Lemma, the Reverse Fatou Lemma and Lebesgue's Dominant Convergence Theorem, respectively. This concludes the proof. 
\end{proof}

\subsection{Second formula --- $\I_u(K)$}

Recall from Definition \ref{def:intro-u-finite} that a compact set $K$ in $\R^n$ is called $u$-finite for a given $u \in \S^{n-1}$, if for a.e.~$y \in u^{\perp}$, $K \cap (y + \sspan u)$ consists of a finite disjoint union of closed intervals of positive length. We will see in Section \ref{sec:Lip} that a Lipschitz star body is necessarily $u$-finite for a.e.~$u \in \S^{n-1}$. For a $u$-finite compact set $K$, we now show that the limit in (\ref{eq:I0}) when $f = 1_K$ exists, and obtain an explicit expression for it. For the reader's convenience, we repeat the formulation of Theorem \ref{thm:intro-Iu} below. 

\begin{thm} \label{thm:Iu}
Let $K$ be a $u$-finite compact set in $\R^n$. 
Then the limit in (\ref{eq:I0}) for $f = 1_K$ exists and $\I_0(K) = \I_u(K)$, where
\begin{equation} \label{eq:Iu}
\I_u(K) := \frac{2}{n}\int_{(P_{u^\perp} K)^n} \Delta(\tilde y_1,\ldots,\tilde y_{n-1})^{-1} |R_\y \cap \theta_\y^{\perp}|_{n-1} \, dy_1 \ldots dy_n .
\end{equation}
Here $\y = (y_1,\ldots,y_n) \in (u^{\perp})^n$, $R_\y = \{ (s^1,\ldots,s^n)  \in \R^n :  y_i + s^i u \in K \, , \,  i=1,\ldots,n\}$, $\theta_\y \in \S^{n-1}$ denotes a linear dependency satisfying $\sum_{i=1}^n \theta_\y^i y_i = 0$, and $(\tilde y_1,\ldots,\tilde y_{n-1})$ denote the $n-1$ rows of the $(n-1) \times n$ matrix whose $n$ columns in $u^{\perp}$ are $(y_1,\ldots,y_n)$. 
\end{thm}

Here the integration in each $dy_i$ is of course with respect to the $\H^{n-1}$ measure on $u^{\perp}$, and all references to a.e.~$\y=(y_1,\ldots,y_n) \in (P_{u^{\perp}} K)^n \subset (u^\perp)^n$ are with respect to the corresponding product measure. 
Note that for a.e.~$\y=(y_1,\ldots,y_n)$, $\{y_1,\ldots,y_n\}$ are affinely independent, and hence the linear dependency $\theta_\y \in \S^{n-1}$  above is unique up to sign and a Borel measurable function of $\y$, and so the above integral is well defined. Instead of using $\theta_\y^{\perp}$ in (\ref{eq:Iu}), we could write $\sspan(\tilde y_1,\ldots,\tilde y_{n-1})$,
as these coincide for a.e.~$\y$, but the present form is more convenient. 

\begin{proof}
Complete $u$ to an orthonormal basis $\mathcal{B} = \{v_1,\ldots,v_{n-1}, u\}$ of $\R^n$. Given 
$\y = (y_1,\ldots,y_n) \in (P_{u^\perp} K)^n$ and $s = (s^1,\ldots,s^n) \in \R^n$, 
let $\Y_s$ denote the $n\times n$ matrix whose $i$-th column is $y_i + s^i u$ written in the basis~$\mathcal{B}$. By definition, the rows of $\Y_s$ are precisely $(\tilde y_1,\ldots,\tilde y_{n-1},s)$, and hence $\Delta(y_1 + s^1 u, \ldots, y_n + s^n u) = \abs{\det(\Y_s)} = \Delta(\tilde y_1, \ldots , \tilde y_{n-1}, s)$. 
In addition, $\theta_\y$  is perpendicular to $\{\tilde y_j\}_{j=1,\ldots,n-1}$. Consequently,
\begin{align}
\nonumber \overline{\I}_0(K) & = \limsup_{p \rightarrow -1^+} \frac{p+1}{n} \int_{K^n} \Delta(x_1,\ldots,x_n)^p \, dx_1 \ldots dx_n \\
\nonumber & = \limsup_{p \rightarrow -1^+} \frac{p+1}{n} \int_{(P_{u^{\perp}} K)^n} \int_{R_{\y}} \Delta(y_1 + s^1 u, \ldots, y_n + s^n u)^p \, ds \, d\y  \\
\nonumber  &= \limsup_{p \rightarrow -1^+} \frac{p+1}{n}\int_{(P_{u^{\perp}} K)^n} \int_{R_{\y}} \Delta(\tilde y_1, \ldots , \tilde y_{n-1}, s)^p \, ds \, d\y \\
\label{eq:DCT2} & = \limsup_{p \rightarrow -1^+} \frac{p+1}{n} \int_{(P_{u^{\perp}} K)^n}  \Delta(\tilde y_1, \ldots , \tilde y_{n-1})^p \int_{R_{\y}}  \abs{\scalar{\theta_\y,s}}^p\;  ds \, d\y .
\end{align}
If we could exchange limit and integration, we could then proceed as follows:
\begin{equation} \label{eq:DCT2-end}
= \frac{2}{n} \int_{(P_{u^{\perp}} K)^n}  \Delta(\tilde y_1, \ldots , \tilde y_{n-1})^{-1} \limsup_{p \rightarrow -1^+} \frac{p+1}{2}\int_{R_{\y}}  \abs{\scalar{\theta_\y,s}}^p ds \, d\y . 
\end{equation}
Since $K$ is assumed to be $u$-finite, we know that $R_\y$ is a finite disjoint union of compact rectangles with non-empty interior for a.e.~$\y \in (P_{u^{\perp}} K)^n$. 
An axis-aligned rectangle $R$ in $\R^n$ trivially satisfies $|\partial R \cap \theta^{\perp}|_{n-1} = 0$ unless $\theta \in \{ \pm e_i \}_{i=1,\ldots,n}$. Since $\theta_\y \neq \{ \pm e_i \}$ unless $y_i = 0$, we conclude that $R_\y$ is a continuity set for the measure $\H^{n-1}|_{\theta_\y^{\perp}}$ for a.e.~$\y \in  (P_{u^{\perp}} K)^n$. By Lemma \ref{lem:weak-convergence}, it follows that for a.e.~$\y \in  (P_{u^{\perp}} K)^n$,
\[
\lim_{p \rightarrow -1^+} \frac{p+1}{2}\int_{R_{\y}}  \abs{\scalar{\theta_\y,s}}^p \, ds =  |R_\y \cap \theta_\y^{\perp}|_{n-1} .
\]
Plugging this into (\ref{eq:DCT2-end}) would then verify that the outer limit exists and complete the proof of (\ref{eq:Iu}). 

It remains to justify exchanging limit and integration in (\ref{eq:DCT2}) by invoking Lebesgue's Dominant Convergence Theorem. Let $R > 0$ be large enough so that $K \subset B_n(R)$, and hence $R_{\y} \subset R Q_n \subset B_n(R \sqrt{n})$ for all $\y \in  (P_{u^{\perp}} K)^n$, where $Q_n = [-1,1]^n = B_\infty^n$. Then, for every $p > -1$,
\[
\frac{p+1}{2} \int_{R_\y} \abs{\scalar{\theta_\y,s}}^p \, ds \leq \frac{p+1}{2} \int_{B_n(R \sqrt{n})} \abs{\scalar{\theta_\y,s}}^p \, ds .
\]
The right-hand side is independent of $\theta_\y$, continuous in $p \in (-1,0]$, and converges as $p \rightarrow -1^+$ to $|B_n(R \sqrt{n}) \cap e_1^{\perp}|_{n-1} < \infty$; consequently, it is bounded by some constant $C_{n,R}$ uniformly in $p \in (-1,0]$. 
In addition, since $t^\alpha \leq 1 + t$ for all $t \geq 0$ and $\alpha \in [0,1]$, we have $\Delta(\tilde y_1,\ldots,\tilde y_{n-1})^p \leq 1 + \Delta(\tilde y_1,\ldots,\tilde y_{n-1})^{-1}$ for all $p \in [-1,0]$. Consequently, for a.e.~$\y \in  (P_{u^{\perp}} K)^n$ and all $p \in (-1,0]$,
\[
\Delta(\tilde y_1,\ldots,\tilde y_{n-1})^p \frac{p+1}{2} \int_{R_\y} \abs{\scalar{\theta_\y,s}}^p \;  ds \leq C_{n,R} (1 + \Delta(\tilde y_1,\ldots,\tilde y_{n-1})^{-1}) . 
\]
It remains to show that the right-hand side is integrable over $(P_{u^{\perp}} K)^n$. Note that $P_{u^{\perp}} K \subset B_{u^{\perp}}(R) \subset R Q_{u^{\perp}}$, where $Q_{u^{\perp}} = Q_{u^{\perp}}(\mathcal{B})$ denotes the unit cube in $u^{\perp}$ in the $\mathcal{B}$-basis. When the columns of the matrix $\Y_0$ range over $R Q_{u^{\perp}}$, the first $n-1$ rows of $\Y_0$ range over $R Q_n$. Consequently, applying a change of variables and the Blaschke--Petkantschin formula (\ref{eq:BP}), we obtain 
\begin{align*}
\int_{(P_{u^{\perp}} K)^n} \Delta(\tilde y_1,\ldots,\tilde y_{n-1})^{-1} \, d\y & \leq \int_{(R Q_{u^{\perp}})^n} \Delta(\tilde y_1,\ldots,\tilde y_{n-1})^{-1} \, d\y \\
& = \int_{(R Q_n)^{n-1}} \Delta(\tilde y_1,\ldots,\tilde y_{n-1})^{-1} \, d\tilde y_1 \ldots d\tilde y_{n-1} \\
& =  \int_{G_{n,n-1}} \int_{(R Q_n \cap E)^{n-1}} \, d\tilde y_1 \ldots d\tilde y_{n-1} \, dE \\
& = \int_{G_{n,n-1}} |R Q_n \cap E|^{n-1} \, dE < \infty . 
\end{align*}
This concludes the proof. 
\end{proof}

\section{Continuous Steiner symmetrization} \label{sec:Steiner}

\subsection{Steiner symmetrization} \label{subsec:Steiner}

Let $u \in \S^{n-1}$,  and recall our notation $L_u = \sspan u$ and $L_u^y = y + L_u$ for $y \in u^{\perp}$. 
Given a compact set $K$ in $\R^n$, its Steiner symmetral $S_u K$ is defined by replacing for every $y \in P_{u^{\perp}} K$ the one-dimensional fiber $K \cap L_u^y$  by a symmetric closed interval in $L_u^y$ having the same one-dimensional Lebesgue measure. In other words,
\[
S_u K \cap L_u^y = y + [-\h_y,\h_y] u \;\; , \;\;  \h_y = \frac{1}{2} |K \cap L_u^y|_1 \;\;\;  \forall y \in P_{u^{\perp}} K 
\]
(and $S_u K \cap L_u^y = \emptyset$ for $y \notin P_{u^{\perp}} K$). 
It is well known that the resulting $S_u K$ remains compact \cite[Proposition 7.1.4]{KrantzParks-GeometryBook}. It is also possible to extend this definition to general Borel sets, but this requires caution since the resulting symmetral may not be a Borel set, only Lebesgue measurable (see \cite[Remark 7.1.6]{KrantzParks-GeometryBook} and \cite[Theorem 2.3]{EvansGariepy-BookRevised}); we refrain from this unnecessary generality here. 

By passing to a layer-cake representation, the definition of Steiner symmetrization immediately extends to very general functions on $\R^n$. Since we only consider the Steiner symmetrization of compact sets, we restrict to upper semi-continuous compactly supported non-negative functions $f \in \usc_c(\R^n,\Real_+)$, since for such functions $\{ f \geq t \}$ is a compact set for all $t > 0$. Writing $f(x) = \int_0^\infty 1_{\{f \geq t\}}(x) \, dt$, we define
\[
S_u f(x) := \int_0^\infty 1_{S_u \{ f \geq t\}}(x) \, dt . 
\]
Since $S_u \{ f \geq t\}$ is compact for all $t > 0$, it follows that $S_u f \in \usc_c(\R^n,\Real_+)$. It is known that if $f \in C_c(\R^n, \R_+)$ then $S_u f \in C_c(\R^n , \R_+)$ \cite[Theorem 6.10]{Baernstein-Book}. 
Note that unlike some authors (e.g., \cite{Baernstein-Book}), we consider the level set $\{ f \geq t\}$ instead of $\{ f > t \}$, which alters the direction of various convergence statements for $\{S_u f_k\}$ in the literature, but the adaptation to our convention is straightforward. If $\{f_k\} \subset \usc_c(\R^n,\R_+)$ is such that $f_k \searrow f$ then of course $f \in \usc_c(\R^n,\R_+)$, and it is known that $S_u f_k \searrow S_u f$ (see \cite[Propositions 1.39, 1.43, 6.3]{Baernstein-Book} and recall that the direction of monotonicity of the convergence should be reversed). 

\smallskip

A function $F : \R^N \rightarrow \R_+$ is called quasi-concave if its super level sets $\{F \geq t\}$ are convex for all $t \geq 0$. 
A functional $G : (\R^n)^N \rightarrow \R_+$ is called Steiner concave if for every $u \in \S^{n-1}$ and $\y = (y_1,\ldots,y_N) \in (u^{\perp})^N$, the function $G_{u,\y} : \R^N \rightarrow \R_+$ given by
\[
G_{u,\y}(t) = G(y_1 + t^1 u,\ldots, y_N + t^N u) 
\]
is even and quasi-concave. It is known and easy to check that $\Delta(x_1,\ldots,x_n)^p$ is Steiner concave for all $p < 0$. We refer to the excellent survey \cite{PaourisPivovarov-Survey} for all references and additional details. 
A very useful property of Steiner concave functionals $G$ is that they cannot decrease under Steiner symmetrization:
\begin{align*}
& \int_{(\R^n)^N} G(x_1,\ldots,x_N) f_1(x_1)\ldots f_N(x_N) \, dx_1 \ldots dx_N \\
& \leq \int_{(\R^n)^N} G(x_1,\ldots,x_N) S_u f_1(x_1)\ldots S_u f_N(x_N) \, dx_1 \ldots dx_N  . 
\end{align*}

While we do not require this for the sequel, we can now easily deduce the following extended version of Corollary \ref{cor:intro-SuK} from the Introduction.\begin{proposition}
Let $f \in \usc_c(\R^n,\R_+)$. Then, for all $u \in \S^{n-1}$,
\[
\overline{\I}_0(f) \leq \overline{\I}_0(S_u f) \text{ and } |I(f)| \leq |I(S_u f)|.
\] 
\end{proposition} 
\begin{proof}
Whenever $p < 0$, $\Delta(x_1,\ldots,x_n)^p$  appearing in (\ref{eq:I0}) is a Steiner concave function, and hence the integral cannot decrease under Steiner symmetrization. The same holds after taking the $\limsup$ as $p \rightarrow -1^+$ and so $\overline{\I}_0(f) \leq \overline{\I}_0(S_u f)$ for any $f \in \usc_c(\R^n, \R_+)$. If $f \in C_c(\R^n,\R_+)$ then $S_u f \in C_c(\R^n,\R_+)$, and hence by Theorem \ref{thm:I0},
\[
|I(f)| = \I_0(f) \leq   \I_0(S_u f) = |I(S_u f)| . \]
To obtain the inequality between the left- and right-most terms for general $f \in \usc_c(\R^n, \R_+)$, we apply Baire's theorem, stating that there exists a sequence $\{f_k\} \subset C_c(\R^n, \R_+)$ such that
$f_k \searrow f$ pointwise. For this sequence, we know that
\[
|I(f_k)| \leq |I(S_u f_k)| . 
\]
As explained above, it is known that $S_u f_k \searrow S_u f$ pointwise. It remains to note that if $\{g_k\}$ are uniformly bounded Borel functions supported in a common compact set which converge pointwise to $g$ then $|I(g_k)|$ converges to $|I(g)|$ by Lebesgue's Dominant Convergence Theorem.
\end{proof}

\subsection{Continuous version on multi-graphical sets} \label{subsec:StK}

When $K$ is a convex body, ensuring that $K \cap L_u^y$ is a compact interval $[c_y-\h_y ,c_y + \h_y]$, an obvious continuous version of Steiner symmetrization $\{S^t_u K\}_{t \in [0,1]}$ may be defined as
\begin{equation} \label{eq:SteinerForConvex}
S^t_u K \cap L_u^y = y + ((1-t) c_y + [-\h_y,\h_y]) u \;\; , \;\;  \h_y = \frac{1}{2} |K \cap L_u^y|_1 \;\;\;  \forall y \in P_{u^{\perp}} K 
\end{equation}
(with $S^t_u K \cap L_u^y = \emptyset$ for $y \notin P_{u^{\perp}} K$). 
For an illustration, see Figure~\ref{fig:contS}.
This is a particular case of a shadow system, introduced and studied by Rogers and Shephard \cite{RogersShephard-ShadowSystems,Shephard-ShadowSystems}, which has proved extremely useful in geometric applications and extremization problems. In particular, $S^t_u K$ remains a convex body for all $t \in [0,1]$. 

\begin{figure} \centering
\includegraphics[width=4.5cm]{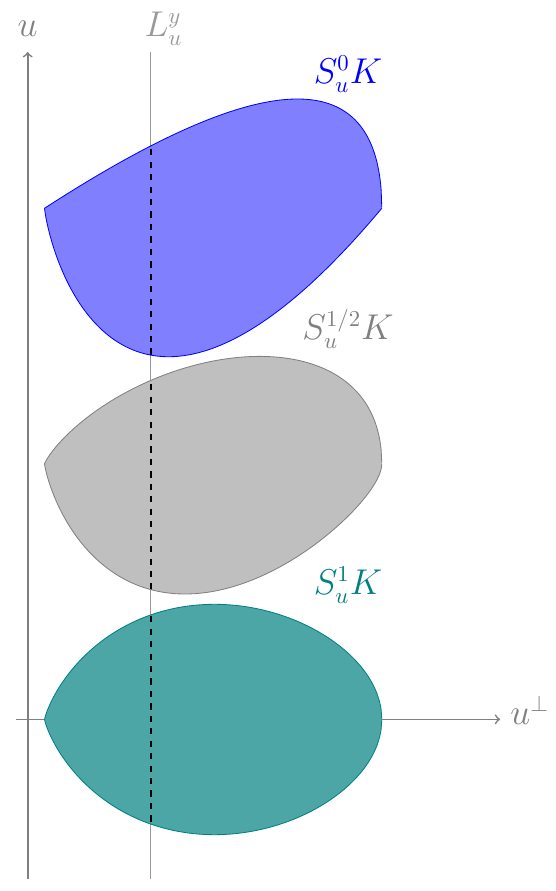}
\caption{An illustration of a continuous Steiner symmetrization
of a convex body $K$ in direction $u$ at three times $0,1/2,1$.
The dashed interval $S^t_u K \cap L_u^y$ of length $2 \h_y$ is being translated towards the origin in $L_u^y$ at a constant velocity (each fiber has its own velocity).}
\label{fig:contS}
\end{figure}

For more general measurable sets $K$ in $\R^n$, various notions of continuous Steiner symmetrization, defined up to null-sets, have been proposed in the literature (see \cite{Brock-ContSteiner, Brock-ContSteiner2, McNabb,Solynin-Cont1, Solynin-Cont2} and also \cite{BGGK-PolyaSzego} for a unified treatment).  
However, for a general compact set $K$, we are not aware of a known 
definition of $S^t_u K$ which leads to a well-defined (not up to null-sets!) compact set on one hand, and which is useful for the geometric applications we have in mind on the other. In this work, we propose such a definition for a certain class of compact sets. 

\begin{defn}[$u$-multi-graphical set] \label{def:mgs} Given $u \in \S^{n-1}$, a compact set $K$ in $\R^n$ is called $u$-multi-graphical, if there exist disjoint open sets $\Omega_1, \Omega_2,\ldots \subset P_{u^{\perp}} K$ and two sequences of
continuous functions
\[
f_i, g_i : \bigcup_{m=i}^\infty \Omega_ m \rightarrow \R ,
\]
such that the following properties hold: \begin{enumerate}
\item \label{it:mgs-1} Denoting $\Omega_\infty := \cup_m \Omega_m$, we have $\H^{n-1}(P_{u^\perp} K \setminus \Omega_\infty) = 0$. 
\item $f_1 < g_1 < f_2 < g_2 < \cdots < f_m < g_m$ on $\Omega_m$. 
\item \label{it:mgs-3} For all $y \in \Omega_m$, $K \cap L_u^y = y + u \cup_{i=1}^m [f_i(y),g_i(y)]$. 
\end{enumerate}
\end{defn}

\begin{rem} \label{rem:mgs}
Since the functions $\{f_i,g_i\}_{i=1,\ldots,m}$ are continuous on the open $\Omega_m$, it follows that $\partial K \cap L_u^y = y + u \cup_{i=1}^m \{f_i(y),g_i(y)\}$ for all $y \in \Omega_m$. In addition, note that a $u$-multi-graphical set $K$ is trivially $u$-finite (recall Definition \ref{def:intro-u-finite}). 
\end{rem}

To define $S^t_u K$ for a $u$-multi-graphical compact set $K$ in $\R^n$, we first define $S^t J$ when $J \subset \R$ is a finite disjoint union of closed intervals $J = \cup_{i=1}^m [c_i-\h_i,c_i + \h_i]$ ($\h_i > 0$) --- we denote the collection of such sets by $\J_m$. The idea, going back to the work of Rogers \cite{Rogers-BLL} and Brascamp--Lieb--Luttinger \cite{BrascampLiebLuttinger}, is as follows. Each interval $[c_i - \h_i,c_i+\h_i]$ is moved independently towards the origin at a constant speed of $-c_i$ until the first time $\tau \in (0,1)$ at which two intervals touch (if there is only one interval set $\tau = 1$). In other words, we define
\[
S^t J = \cup_{i=1}^m ((1-t) c_i + [-\h_i,\h_i]) \;\;\; t \in [0,\tau] . 
\]
If $\tau < 1$, this means that at time $\tau$ the number of intervals $m'$ in $S^{\tau} J = \cup_{i=1}^{m'} [c'_i - \h'_i,c'_i + \h'_i]$ has decreased, and we recursively set
\begin{equation} \label{eq:semi-group}
S^t J = S^{\frac{t-\tau}{1-\tau}} (S^{\tau} J) \;\;\; t \in [\tau , 1] . 
\end{equation}
Clearly $|S^t J|_1 = |J|_1$ for all $t \in [0,1]$ and $S^1 J = [-\frac{|J|_1}{2} , \frac{|J|_1}{2}]$. It is also easy to check 
that the ``semi-group" property (\ref{eq:semi-group}) remains valid for all $\tau \in [0,1]$ (interpreting $0/0$ as $0$); this is easier to see using an alternative time parametrization $[0,\infty] \ni s \mapsto t = 1 - e^{-s} \in [0,1]$ and setting $\bar S^s = S^t$, whence (\ref{eq:semi-group}) becomes
\[ \bar S^{s_1 + s_2} J = \bar S^{s_2} \bar S^{s_1} J \;\;\; \forall s_1,s_2 \in [0,\infty] . 
\] 
After a preliminary version of this work was completed, we learned from G.~Bianchi and R.~Gardner about Brock's work \cite{Brock-ContSteiner,Brock-ContSteiner2}, where he uses the $\bar S^s$ parametrization to define the continuous symmetrization $\bar S^s J$, first for $J \in \J_m$ and then up to null-sets for general measurable subsets $J \subset \R$ (of finite Lebesgue measure). Brock then applies this operation fiberwise to extend his definition to $\R^n$, but it is not clear why this would preserve compactness, nor how to describe the boundary of the resulting sets $\bar S^s_u K$. In contrast, our idea is to only apply $S^t$ to the ``good fibers" over $\Omega_\infty$ and then take the closure of the resulting set, but we then need to justify that this does not alter the action of $S^t$ on the good fibers; consequently, our construction is restricted to $u$-multi-graphical sets where this can be ensured. To this end, we require several simple lemmas. Lemmas \ref{lem:StMonotone} and \ref{lem:St-sum} below were also obtained by Brock, but for completeness and to keep our presentation self-contained, we have left our original proofs as they first appeared. The other statements below, in particular those regarding control of and convergence in the Hausdorff distance, as well as preservation of star-shapedeness, appear to be new. 

\begin{lemma}[Monotonicity] \label{lem:StMonotone}
Let $J^1 \in \J_{m_1}$ and $J^2 \in \J_{m_2}$. If $J^1 \subseteq J^2$ then $S^t J^1 \subseteq S^t J^2$ for all $t \in [0,1]$.
\end{lemma}
\begin{proof}
Since $S^t$ verifies the semi-group property (\ref{eq:semi-group}), then inducting on $m_1+m_2$, it is enough to prove that $S^t J^1 \subseteq S^t J^2$ for all $t \in [0,\min(\tau_1,\tau_2)]$, where $\tau_i$ is the first collision time for $S^t J^i$, because at that time the number of total intervals strictly decreases. Since until the first collision, each interval evolves independently of others, we further reduce to the case $m_1 = m_2 = 1$ and $\tau_1 = \tau_2 = 1$ which is the base of the induction. But this case is trivial: we are given that $J^1 \subseteq J^2$ and therefore $S^1 J^1 = \frac{1}{2} [-|J^1| , |J^1|] \subseteq \frac{1}{2} [-|J^2| ,|J^2|] = S^1 J^2$, and since $S^t J^i = (1-t) J^i + t S^1 J^i$, we conclude that $S^t J^1 \subseteq S^t J^2$ for all $t \in [0,1]$. 
\end{proof}

For $J \in \J_m$, we denote by $\{J_i\}_{i=1,\ldots,m}$ the individual intervals comprising $J$. We say that $J^1 \in \J_{m_1}$ entwines $J^2 \in \J_{m_2}$ if $J^1$ intersects each interval comprising $J^2$. 
For compact subsets $A,B \subset \R$, we denote by $A_\eps$ the compact set $A + [-\eps,\eps]$, and their Hausdorff distance by $d_H(A,B) := \inf \{\eps > 0 : A \subseteq B_\eps , B \subseteq A_\eps \}$. 

\begin{lemma} \label{lem:StHausdorff}
Let $J^1 \in \J_{m_1}$, $J^2 \in \J_{m_2}$ and $t \in [0,1]$. 
\begin{enumerate}
\item If $S^t J^2 \in \J_{m'_2}$, then for all $i=1,\ldots,m'_2$ there exists $j=1,\ldots,m_2$ such that $(S^t J^2)_i \supseteq S^t(J^2_j)$. 
\item If $J^1$ entwines $J^2$ then $S^t J^1$ entwines $S^t J^2$. 
\item If $J^1 \subseteq J^2$ and $J^1$ entwines $J^2$, then $S^t J^2 \subseteq  (S^t J^1)_\delta$, where $\delta = |J^2|-|J^1|$.
\item In particular, $S^t J^1_\eps \subseteq (S^t J^1)_{2 m_1 \eps}$ for all $\eps > 0$.
\item If $d_H(J^1,J^2) \leq \eps$ then $d_H(S^t J^1,S^t J^2) \leq 2 \max(m_1,m_2) \eps$.
\end{enumerate}
\end{lemma}
\begin{proof} \hfill
\begin{enumerate}
\item We verify the claim by induction on $m_2$, with the case $m_2 = 1$ being trivial. Since until the first collision time $\tau$ the intervals comprising $J^2$ evolve independently, the claim holds trivially for $t \in [0,\tau)$, and also at $t = \tau$ as some intervals get merged. Since $S^\tau  J^2 \in \J_{m_2'}$ with $m_2' < m_2$, by the induction hypothesis $(S^{t'} S^\tau J^2)_i \supseteq S^{t'}( (S^\tau J^2)_k)$ for some $k$, where we denote $t' = \frac{t-\tau}{1-\tau}$. But $(S^\tau J^2)_k \supseteq S^\tau(J^2_j)$ for some $j$, and so by monotonicity $S^{t'}( (S^\tau J^2)_k) \supseteq S^{t'} S^\tau (J^2_j)$. By (\ref{eq:semi-group}), we conclude that $(S^t J^2)_i \supseteq S^t(J^2_j)$, as required.
\item By the first part, given $(S^t J^2)_i$, there exists $j$ such that $(S^t J^2)_i \supseteq S^t(J^2_j)$. Since $J^1$ entwines $J^2$, $J^3 = J^1 \cap J^2_j \neq \emptyset$. By monotonicity,  $\emptyset \neq S^t J^3 \subseteq S^t J^1 \cap S^t (J^2_j) \subseteq S^t J^1 \cap (S^t J^2)_i$. This shows that $S^t J^1$ entwines $S^t J^2$. 
\item
By the previous part, every interval comprising $S^t J^2$ intersects $S^t J^1$. Note that if $I$ is a compact interval and $A$ is a non-empty compact subset of $I$, then $I \subseteq A_\delta$ for $\delta = |I \setminus A| = |I| - |A|$ (since otherwise, there exists $x \in I$ with $|x-a| = d(x,A) > \delta$, and therefore $|I| \geq |A| + |[x,a]| > |I|$, a contradiction). Applying this to $I = (S^t J^2)_i$ and $A = S^t J^1 \cap (S^t J^2)_i \neq \emptyset$,  since $S^t J^1 \subseteq S^t J^2$ by monotonicity, we have
\[
|(S^t J^2)_i \setminus (S^t J^1 \cap (S^t J^2)_i)| \leq |S^t J^2 \setminus S^t J^1| = |S^t J^2| - |S^t J^1| = |J^2| - |J^1| = \delta ,
\]
and hence $(S^t J^2)_i  \subseteq (S^t J^1 \cap (S^t J^2)_i)_\delta$. Taking the union over $i$, we obtain
\[
S^t J^2 \subseteq \cup_i (S^t J^1 \cap (S^t J^2)_i)_\delta = (S^t J^1 \cap \cup_i (S^t J^2)_i)_\delta = (S^t J^1)_\delta . 
\]
\item Clearly $|J^1_\eps| - |J^1|\leq 2 m_1 \eps$, and since $J^1$ trivially entwines $J^1_\eps$, we have by the previous part that $S^t J^1_\eps \subseteq (S^t J^1)_{2 m_1 \eps}$. 
\item If $J^2 \subseteq J^1_\eps$ then by monotonicity and the previous part $S^t J^2 \subseteq S^t J^1_\eps \subseteq (S^t J^1)_{2 m_1 \eps}$. Exchanging the roles of $J^1,J^2$, the assertion follows. 
\end{enumerate}
\end{proof}

Identifying $L_u^y$ and $\R$, the definition of $S^t$ extends to finite disjoint unions of closed intervals in $L_u^y$. 
\begin{cor} \label{cor:fibers-converge}
If $y_k \rightarrow y \in \Omega_\infty$ then $S^t (K \cap L^{y_k}_u) \rightarrow S^t (K \cap L^{y}_u)$ in the Hausdorff metric for all $t \in [0,1]$. 
\end{cor}
\begin{proof}
Let $\eps > 0$. If $y \in \Omega_m$, there exists $\delta > 0$ such that for all $y' \in B_{u^{\perp}}(y,\delta) \subset \Omega_m$, $\abs{f_i(y') - f_i(y)} , \abs{g_i(y') - g_i(y)} < \eps$ for all $i=1,\ldots,m$ (as these functions are all continuous). Consequently, $d_H(K \cap L^{y'}_u, K \cap L^y_u) < \eps$, and hence by Lemma \ref{lem:StHausdorff} we have $d_H(S^t (K \cap L^{y'}_u), S^t (K \cap L^{y}_u)) \leq 2 m \eps$ (here we identify both $L^{y}_u$ and $L^{y'}_u$ with $\R$ when evaluating the Hausdorff distance). As $\eps > 0$ was arbitrary, this concludes the proof. 
\end{proof}

We can now give the following definition.\begin{defn}[Continuous Steiner symmetrization of a $u$-multi-graphical compact set $K$] \label{def:SutK-general}
\[
\mathring S^t_{u,\{\Omega_m\}} K := \cup_{y \in \Omega_\infty} S^t (K \cap L^y_u)  ~,~ S_u^t K  := \closure(\mathring S_{u,\{\Omega_m\}}^t K) .
\]
\end{defn}
Note that the definition of $\mathring S^t_{u,\{\Omega_m\}} K$ depends on the particular choice of open sets $\{\Omega_m\}$ in the $u$-multi-graphical representation of $K$, but when this choice is fixed we will simply abbreviate by $\mathring S^t_u K$. Furthermore, $\mathring S^t_u K$ is not a closed set. In contrast, $S_u^t K$ does not possess these two caveats: it is trivially closed (and hence compact), and in addition
the following holds.
\begin{prop} \label{prop:independent}
For all $t \in [0,1]$, $S_u^t K$ does not depend on the particular choice of $\{ \Omega_m \}$ in its $u$-multi-graphical representation. 
\end{prop}
We will prove the following more general statement.
\begin{lem} \label{lem:closure}
Let $\{\Omega'_m\}$ denote another sequence of open sets satisfying the requirements in Definition \ref{def:mgs}. 
Then, for all $t \in [0,1]$ and $y \in \Omega_\infty$,
\[
\mathring S^t_{u,\{\Omega_m\}} K \cap L_u^y = \closure(\mathring S^t_{u , \{\Omega'_m\}} K) \cap L_u^y .
\]
\end{lem}
\begin{proof}
If $\{ f'_i, g'_i \}$ are the sequences of continuous functions corresponding to $\{\Omega'_m\}$, then by property (\ref{it:mgs-3}) of Definition \ref{def:mgs}, $f_i,g_i$ coincide with $f'_i,g'_i$ on $\Omega_m \cap \Omega'_m$ for all $m \geq i$. 
By property (\ref{it:mgs-1}), $\Omega'_\infty$ is dense in $P_{u^\perp} K$ and in particular in $\Omega_\infty$. Let $\{y_k\} \subset \Omega'_\infty$ be any sequence converging to $y \in \Omega_\infty$. By Corollary \ref{cor:fibers-converge}, we see that $S^t (K \cap L^{y_k}_u)$ converges to $S^t(K \cap L^y_u)$ in the Hausdorff metric, thereby concluding the proof. 
\end{proof}
\begin{proof}[Proof of Proposition \ref{prop:independent}]
By Lemma \ref{lem:closure}, taking the union over all $y \in \Omega_\infty$ we have:
\[
\mathring S^t_{u,\{\Omega_m\}}  \subseteq \closure(\mathring S^t_{u , \{\Omega'_m\}} K) .
\]
Taking the closure of the left-hand side and reversing the roles of $\{\Omega_m\}$ and $\{\Omega'_m\}$, we confirm that both closures coincide. 
\end{proof}

Since the choice of $\{\Omega_m\}$ makes no difference, we revert back to our abbreviated notation and restate Lemma \ref{lem:closure} as follows: 
\begin{cor} \label{cor:closure}
For all $y \in \Omega_\infty$ and $t \in [0,1]$, $S_u^t K \cap L_u^y = \mathring S_u^t K \cap L_u^y$. 
\end{cor}

\begin{cor} \label{cor:u-finite}
For all $t \in [0,1]$, $S_u^t K$ is $u$-finite. 
\end{cor}

By Fubini's Theorem and property (\ref{it:mgs-1}) of Definition \ref{def:mgs}, we immediately deduce the following corollary.
\begin{cor} \label{cor:StVolume}
For all $t \in [0,1]$, $|S_u^t K| = |\mathring S_u^t K| = |K|$. 
\end{cor}

In addition, since both pairs $\mathring S_u^0 K \subset K$ and $\mathring S_u^1 K \subset S_u K$ coincide on $\Omega_\infty \times L_u$, and $K$ and $S_u K$ are closed, we deduce the following.
\begin{cor} \label{cor:coincide}
Both pairs $S_u^0 K \subseteq K$ and $S_u^1 K \subseteq S_u K$ coincide up to an $\H^n$-null set. 
\end{cor}

So at least up to null-sets, $S_u^t K$ is indeed a continuous-time version of the classical Steiner symmetrization (for $u$-multi-graphical sets $K$). 
We also record the following.
\begin{cor} \label{cor:StBalls}
If $K_1 \subseteq K_2$ are both $u$-multi-graphical, then $S_u^t K_1 \subseteq S_u^t K_2$ for all $t \in [0,1]$. In particular, if $B_n(r) \subseteq K \subseteq B_n(R)$ then $B_n(r) \subseteq S^t_u K \subseteq B_n(R)$ for all $t \in [0,1]$. 
\end{cor}
\begin{proof}
Monotonicity implies that for all $t \in [0,1]$,
\begin{equation} \label{eq:StBalls}
\mathring S_u^t K_1 \cap ((\Omega^1_\infty \cap \Omega^2_\infty) \times L_u) \subseteq \mathring S_u^t K_2 , \end{equation}
where $\{\Omega^i_m\}$ and $\Omega^i_\infty$ are the sets from Definition \ref{def:mgs} corresponding to $K_i$. Note that $\{\Omega^1_m \cap \Omega^2_\infty\}$ 
also satisfy the requirements of Definition \ref{def:mgs} for $K_1$. Taking the closure in (\ref{eq:StBalls}) and applying Proposition \ref{prop:independent}, the assertion follows. 
\end{proof}

\subsection{Star-shapedness is preserved}

\begin{lemma} \label{lem:St-sum}
Let $J \in \J_m$. Then, for all $a,b \geq 0$ and $t \in [0,1]$,
\[
a \cdot S^t J + [-b,b] \subseteq S^t(a \cdot J + [-b,b]) . 
\]
\end{lemma}
\begin{proof}
There is nothing to prove if $a = 0$, and if $a > 0$, by scaling we may assume that $a=1$. Write $J = \cup_{i=1}^m J_i$ with disjoint compact intervals $J_1,\ldots,J_m$, and let $\tau$ denote the first collision time in $S^t J$. Clearly $S^t J_i + [-b,b] = S^t (J_i + [-b,b])$. Since each interval evolves independently before the first collision, we have for $t \in [0,\tau]$,
\begin{equation} \label{eq:ab}
S^t J + [-b,b] = \cup_{i=1}^m (S^t J_i + [-b,b]) = \cup_{i=1}^m S^t(J_i + [-b,b]) \subseteq S^t(J + [-b,b]) ,
\end{equation}
where the last inclusion is by monotonicity. In particular, this confirms the claim for $m=1$. The general case follows by induction on $m$, since $S^{\tau} J \in \J_{m'}$ for $m' < m$ and hence by (\ref{eq:semi-group}), the induction hypothesis for $S^\tau J$, (\ref{eq:ab}) for $t = \tau$, and monotonicity, we obtain for all $t \in [\tau,1]$,
\begin{align*}
S^t J + [-b,b] & = S^{\frac{t-\tau}{1-\tau}} S^{\tau} J + [-b,b] \subseteq S^{\frac{t-\tau}{1-\tau}} (S^{\tau} J + [-b,b]) \\
& \subseteq S^{\frac{t-\tau}{1-\tau}} S^{\tau}(J + [-b,b]) = S^t(J + [-b,b]) . 
\end{align*}
\end{proof}

\begin{prop} \label{prop:Steiner-star-shaped}
Let $u \in \S^{n-1}$, and let $B$ be any subset of $\R^n$ so that for all $y \in u^{\perp}$, $B \cap L_u^y$ is a compact \emph{symmetric} interval (possibly a singleton or empty). Let $K$ be a $u$-multi-graphical compact subset of $\R^n$ which is star-shaped with respect to $B$. Then $S^t_u K$ remains star-shaped with respect to $B$ for all $t \in [0,1]$. 
\end{prop}
\begin{proof}
It is enough to prove the claim for $B = y + [-b,b] u$, $y \in u^{\perp}$. Since $S_u^t (K-y) = S_u^t K - y$ for all $y \in u^{\perp}$, we reduce to the case that $B = [-b,b] u$. Fix $\lambda \in [0,1]$, and let $y \in \Omega_\infty$ be such that $\lambda y \in \Omega_\infty$ as well. Since $K$ is star-shaped with respect to $B$, we know that $\lambda (K \cap L_u^y) + (1-\lambda) B \subseteq K \cap L_u^{\lambda y}$. By Lemma \ref{lem:St-sum} and monotonicity (Lemma \ref{lem:StMonotone}), we conclude that $\lambda S^t (K \cap L_u^y) + (1-\lambda) B \subseteq S^t (K \cap L_u^{\lambda y})$. In other words,
\[
\brac{\lambda \mathring S_u^t K + (1-\lambda) B} \cap (\Omega_\infty \times L_u) \subseteq \mathring S^t_u K . \]
Since this holds for all $\lambda \in [0,1]$, this means that
\begin{equation} \label{eq:rings}
\forall \beta \in [-b,b] \;\;\; \forall x \in \mathring S_u^t K \;\;\; [\beta u,x] \cap (\Omega_\infty \times L_u) \subseteq \mathring S^t_u K . 
\end{equation}
Setting $\Omega_0 = P_{u^{\perp}} K \setminus \Omega_\infty$, we are given that
\[
0 = \H^{n-1}(\Omega_0) = \int_{\S^{n-1} \cap u^{\perp}} \int_0^\infty 1_{\Omega_0}(r \theta) r^{n-2} \, dr \, d\theta .
\]
It follows that for a.e.~$\theta \in \S^{n-1} \cap u^{\perp}$, $|\Omega_0 \cap \R_+ \theta|_1 = 0$, and in particular, $\Omega_\infty \cap \R_+ \theta$ is dense in $P_{u^{\perp}} K \cap \R_+ \theta$. Taking the closure in (\ref{eq:rings}), we deduce that for a.e.~$\theta \in \S^{n-1} \cap u^{\perp}$, for all $r \geq 0$, $s \in \R$ and $\beta \in [-b,b]$, if $x = r \theta + s u \in \mathring S_u^t K$ then $[\beta u,x] \subset \closure(\mathring S^t_u K) = S^t_u K$. Since the set of such good $\theta$'s is dense in $\S^{n-1} \cap u^{\perp}$, and as $[\beta u,x] \subset \closure(\cup_k [\beta u,x_k])$ if $x_k \rightarrow x$, it follows that for all $x \in  S_u^t K$ and $\beta \in [-b,b]$, $[\beta u,x] \subset S^t_u K$. 
Since $S_u^t K$ is closed, this shows that $S_u^t K$ is star-shaped with respect to $[-b,b] u$, concluding the proof. \end{proof}

\subsection{Lipschitz continuity in time}

We will also need the following in the sequel.

\begin{lem} \label{lem:StLip}
Let $J \in \J_m$, and assume that all of its centers $\{c_i(J)\}_{i=1}^m$ are contained in $[-R,R]$. Then, $d_H(S^t J , J) \leq R t$ for all $t \in [0,1]$.
\end{lem}
\begin{proof}
We will prove the claim by induction on $m$. Note that before the first collision time $\tau$, the intervals comprising $J$ are being translated at a velocity of at most $R$. Consequently, for all $t \in [0,\tau]$, $d_H(S^t J,J) \leq R t$, establishing in particular the claim when  $m=1$ (and hence $\tau=1$). In addition, note that for all $t \in [0,\tau)$, $\{c_i(S^t J)\}_{i=1}^m \subset (1-t) [-R,R]$, and that at the collision time $t=\tau$, the new centers $\{c'_i(S^\tau J) \}_{i=1}^{m'}$ are a convex combination of the old centers, and hence $\{c'_i(S^\tau J)\}_{i=1}^{m'} \subset (1-\tau) [-R,R]$. Since $m' < m$, we may apply the induction hypothesis. Recalling that $S^t J = S^{\frac{t-\tau}{1-\tau}} S^\tau J$, we know that $d_H(S^t J , S^\tau J) \leq (1-\tau) R \frac{t-\tau}{1-\tau} = R (t-\tau)$ for all $t \in [\tau ,1]$. It remains to apply the triangle inequality for the Hausdorff distance, verifying that for all $t \in [\tau,1]$,
\[
d_H(S^t J , J) \leq  d_H(S^t J, S^\tau J) + d_H(S^\tau J , J) \leq  R (t-\tau) + R \tau \leq R t . 
\]
\end{proof}

\section{Lipschitz star bodies} \label{sec:Lip}

It is shown in the appendix that any star body $K$ satisfying $I^2 K = c K$ must have a $C^\infty$-smooth radial function $\rho_K$. For our purposes, there is  no benefit in utilizing any regularity of $\rho_K$ beyond Lipschitzness, 
 and so in this work we will concentrate on Lipschitz star bodies and their properties.

\begin{defn}[Lipschitz star bodies]
A star body $K$ in $\R^n$ is called a Lipschitz star body if its radial function $\rho_K : \S^{n-1} \rightarrow (0,\infty)$ is Lipschitz continuous. 
\end{defn}

Recall that the gauge function $\norm{x}_K$ is the $1$-homogeneous function on $\R^n$ coinciding with $1 / \rho_K(x)$, and that $\norm{x}_K \leq 1$ iff $x \in K$.  
Clearly, $\rho_K$ is Lipschitz and strictly positive on $\S^{n-1}$ iff $\norm{\cdot}_K$ is, and therefore so is the $1$-homogeneous extension of $\norm{\cdot}_K$ to $\R^n$.
In particular, it follows that the class of Lipschitz star bodies includes all convex bodies $K$ containing the origin in their interior, since $\norm{\cdot}_K$ is trivially Lipschitz by the triangle inequality. For a Lipschitz star body $K$, we denote by $\Lip_K < \infty$ the Lipschitz constant of $\norm{\cdot}_K$ on $\R^n$.

\smallskip
The following proposition is known (see \cite{LinWu-LipschitzStarBodiesReview} and the references therein).
\begin{prop} \label{prop:Lip}
Let $K$ be a compact set in $\R^n$ with $B_n(r) \subseteq K \subseteq B_n(R)$. The following statements are equivalent:
\begin{enumerate}
\item \label{it:Lip1} $K$ is a Lipschitz star body with $\Lip_K \leq L$.
\item \label{it:Lip2} There exists $\delta > 0$ such that $K$ is a star body with respect to $B_n(\delta)$. 
\item \label{it:Lip3} There exists $\eps > 0$ such that $K$ is star-shaped with respect to $B_n(\eps)$. 
\end{enumerate}
The equivalence is in the sense that the constants $\eps,\delta,L>0$ above only depend on each other and on $r, R > 0$. 
\end{prop}
\begin{proof}
As explained above, one may pass back and forth between upper bounds on the spherical Lipschitz constant of $\rho_K$ and $\Lip_K$, in a manner depending solely on $r , R$. Consequently, the equivalence between (\ref{it:Lip1}) and (\ref{it:Lip3}) follows from \cite[Theorem 2.1]{LinWu-LipschitzStarBodiesReview} (see \cite[Lemmas 3.2, 3.3 and 3.4]{LinWu-LipschitzStarBodiesReview}); see also \cite[Theorem 2]{Toranzos-LipschitzConstant} for the best dependence of the spherical Lipschitz constant of $\rho_K$ on the inner and outer radii of $K$ in the implication $(\ref{it:Lip3}) \Rightarrow (\ref{it:Lip1})$. 
Clearly (\ref{it:Lip2}) implies (\ref{it:Lip3}) with $\eps = \delta$. The other direction for any $\delta \in (0,\eps)$ follows by \cite[Lemma 3.1]{LinWu-LipschitzStarBodiesReview} and the subsequent comment. 
\end{proof}

\begin{lem} \label{lem:adding-ball}
Let $K$ be a Lipschitz star body in $\R^n$. Then, for all $d > 0$, $K + B_n(d)$ is a Lipschitz star body satisfying
\[
\rho_{K + B_n(d)}(\theta) \leq (1 + \Lip_K d) \rho_K(\theta) \;\;\; \forall \theta \in \S^{n-1} . 
\]
\end{lem}
\begin{proof}
By Proposition \ref{prop:Lip}, $K$ is star-shaped with respect to $B_n(\eps)$. It follows that $K + B_n(d)$ is star-shaped with respect to $B_n(\eps + d)$, since if $x = y + z$ with $y \in K$ and $z \in B_n(d)$, then for any $x_0 \in B_n(\eps + d)$, write $x_0 = y_0 + z_0$ with $y_0 \in B_n(\eps)$ and $z_0 \in B_n(d)$, and note that $[y_0,y] \subset K$ and $[z_0,z] \subset B_n(d)$, and therefore $[y_0+z_0 , y+z] \subset K + B_n(d)$. Consequently, $K_2 = K + B_n(d)$ is a Lipschitz star body by Proposition \ref{prop:Lip}. In addition, we claim that
\begin{equation} \label{eq:inf}
\inf_{z \in B_n(d)} \norm{x - \norm{x}_{K_2} z}_K \leq \norm{x}_{K_2} \;\;\; \forall x \in \R^n . 
\end{equation}
Indeed, both sides are homogeneous in $x$, so it is enough to verify this for $\norm{x}_{K_2} = 1$. This means that $x \in K_2$, and so there exists $z \in B_n(d)$ such that $x - z \in K$, hence $\norm{x-z}_K \leq 1$, and (\ref{eq:inf}) is verified. Therefore,
\[
\norm{x}_K - \Lip_K d \norm{x}_{K_2} \leq \norm{x}_{K_2} \;\;\; \forall x \in \R^n . 
\]
Rearranging and recalling that $\rho(\theta) = \frac{1}{\norm{\theta}}$, this concludes the proof. 
\end{proof}

\subsection{Graphical properties}

\begin{lem} \label{lem:jointly-cont}
Let $K$ be a star body with respect to $B_n(\delta)$ in $\R^n$. Then, the function
\[
\Omega = B_n(\delta) \times \S^{n-1} \ni (x,\theta) \mapsto \rho_{K - x}(\theta) \in (0,\infty) 
\]
is jointly continuous.  \end{lem}
\begin{proof}
Clearly there exists $r > \delta$ such that $B_n(r) \subseteq K$ (otherwise $K-x$ would not be a star body for some $x \in \partial B_n(\delta)$). 
Assume $\Omega \ni (x_k,\theta_k) \rightarrow (x_0,\theta_0) \in \Omega$ as $k\rightarrow \infty$. Let $\theta'_k \in \S^{n-1}$ be the direction in which $x_k + \rho_{K - x_k}(\theta_k) \theta_k - x_0$ is pointing. Since $\rho_{K-x_k}(\theta_k) \geq r-\delta > 0$, this is well defined for large enough $k$, and since $\theta_k \rightarrow \theta_0$, it follows that $\theta'_k \rightarrow \theta_0$ (regardless of whether $\rho_{K - x_k}(\theta_k)$ converges to $\rho_{K-x_0}(\theta_0)$ or not). Since $K-x_0$ is a star body, this implies that $\rho_{K-x_0}(\theta'_k) \rightarrow \rho_{K-x_0}(\theta_0)$, and therefore
\[
 x_k + \rho_{K-x_k}(\theta_k) \theta_k = x_0 + \rho_{K-x_0}(\theta'_k)  \theta'_k \rightarrow x_0 + \rho_{K-x_0}(\theta_0) \theta_0 . 
 \]
 Since $x_k \rightarrow x_0$ and $\theta_k \rightarrow \theta_0$, it follows that $\rho_{K-x_k}(\theta_k) \rightarrow \rho_{K-x_0}(\theta_0)$. This concludes the proof. 
\end{proof}

\begin{defn}[$u$-graphical and equi-graphical]
Given $u \in \S^{n-1}$, we say that a compact set $K$ in $\R^n$ is $u$-graphical over a subset $\Omega_u \subset u^{\perp}$ if
\begin{equation} \label{eq:u-graphical}
K \cap (\Omega_u \times L_u) = \set{ y+ s u : y \in \Omega_u \; ,\; f(y) \leq s \leq g(y) } ,
\end{equation}
for some continuous functions $f < g : \Omega_u\rightarrow \R$. 

\medskip

We say that $K$ is equi-graphical over $\Omega \subset \R^n$ if for all $u \in \S^{n-1}$, $K$ is $u$-graphical over $\Omega \cap u^{\perp}$, and moreover, the corresponding graph functions $f_u,g_u$ satisfy $g_u(y) = F(y,u)$ and $f_u(y) = -F(y,-u)$ for some common uniformly continuous function $F : \Omega \times \S^{n-1} \rightarrow \R$. 
\end{defn}

\begin{prop} \label{prop:graphical}
Let $K$ be a Lipschitz star body in $\R^n$. There exists $\delta > 0$ (with $B_n(\delta) \subset \interior K$) such that $K$ is equi-graphical over $B_n(\delta)$. 
\end{prop}
\begin{proof}
By Proposition \ref{prop:Lip}, $K$ is a star body with respect to $B_n(\delta)$ for some $\delta > 0$, and so $(x,\theta) \mapsto F(x,\theta) = \rho_{K - x}(\theta)$ is uniformly continuous on the compact set $B_n(\delta) \times \S^{n-1}$ by Lemma  \ref{lem:jointly-cont}. 
Given $u \in \S^{n-1}$ and $y \in B_{u^{\perp}}(\delta)$, since $K$ is a star body with respect to $y$, it follows that $K \cap L_u^y$ is a closed interval of the form $y + [f_u(y),g_u(y)] u$ with $f_u(y) < 0 < g_u(y)$, having its endpoints in $\partial K$. Since $g_u(y) = F(y,u)$ and $f_u(y) = -F(y,-u)$, the equi-graphicality is established.
\end{proof}

In addition, the following multi-graphical version of Proposition \ref{prop:graphical} was shown by Lin and Xi \cite[Lemma 2.2, Section 3 and Theorem 4.1]{LinXi-LipschitzStarBodySymmetrization}. Recall the Definition \ref{def:mgs} of a $u$-multi-graphical set. 
\begin{thm}[\cite{LinXi-LipschitzStarBodySymmetrization}] \label{thm:Lip-multi-graphical}
Let $K$ be a Lipschitz star body in $\R^n$. Then, there exists a Lebesgue measurable $\U \subseteq \S^{n-1}$ of full measure such that for all $u \in \U$, $K$ is $u$-multi-graphical (and in particular, $u$-finite), and moreover, the following properties hold:
\begin{enumerate}
\item For all $m$, the corresponding functions $\{f_i,g_i\}_{i=1,\ldots,m}$ from Definition \ref{def:mgs} are differentiable in a Lebesgue measurable subset $\Omega_m^* \subseteq \Omega_m$ with $\H^{n-1}(\Omega_m \setminus \Omega_m^*) =0$, and
\item $\H^{n-1}(\partial K \setminus (\Omega_\infty \times L_u)) = 0$. 
\end{enumerate}
\end{thm}

\subsection{Continuous Steiner Symmetrization of Lipschitz star bodies}

In view of Theorem \ref{thm:Lip-multi-graphical} and the discussion in Subsection \ref{subsec:StK}, the continuous Steiner symmetrization of a Lipschitz star body $K$ is well defined for a.e.~$u \in \S^{n-1}$. In addition, we have the following 
proposition. 
\begin{prop} \label{prop:StRemainsLip}
Let $K$ be a Lipschitz star body in $\R^n$. Then, there exists $L > 0$ such that for any $u \in \S^{n-1}$ for which $K$ is $u$-multi-graphical and for all $t \in [0,1]$, $S_u^t K$ is a Lipschitz star body with $\Lip_{S_u^t K} \leq L$.  
\end{prop}
\begin{proof}
If $B_n(r) \subseteq K \subseteq B_n(R)$, this remains valid for $S_u^t K$ and all $t \in [0,1]$ by Corollary \ref{cor:StBalls}. By Proposition \ref{prop:Lip}, $K$ is star-shaped with respect to $B_n(\delta)$ for some $\delta > 0$, and Proposition \ref{prop:Steiner-star-shaped} ensures that this remains valid too for $S_u^t K$ for all $t \in [0,1]$. Another application of Proposition \ref{prop:Lip} shows that for all $t \in [0,1]$, $S_u^t K$ is a Lipschitz star body with $\Lip_{S_u^t K} \leq L$ depending solely on $r,R,\delta > 0$. 
\end{proof}
\begin{cor} \label{cor:Su01}
Let $K$ be a Lipschitz star body in $\R^n$. Then, for any $u \in \S^{n-1}$ for which $K$ is $u$-multi-graphical, $S_u^0 K = K$ and $S_u^1 K = S_u K$.
\end{cor}
\begin{proof}
By Proposition \ref{prop:StRemainsLip}, both $S_u^0 K$ and $S_u^1 K$ are (Lipschitz) star bodies. In addition, it is known that if $K$ is a star body then so is $S_u K$ \cite[Theorem 3.3]{Zhu-OrliczCentroidInqForStarBodies} (see also \cite[Lemma 5.1]{LinWu-LipschitzStarBodiesReview} for an analogous statement for Lipschitz star bodies). Formula (\ref{eq:vol-polar}) verifies that if $K_1 \subseteq K_2$ are two star bodies with $\H^n(K_2 \setminus K_1) = 0$ then continuity of $\rho_{K_i}$ implies $K_1 = K_2$, and so applying this to 
the nested pairs $S_u^0 K \subseteq K$ and $S_u^1 K \subseteq S_u K$ and recalling Corollary \ref{cor:coincide}, we conclude that $S_u^0 K = K$ and $S_u^1 K = K$.  
\end{proof}
\begin{rem}
Note that a convex body $K$ is $u$-multi-graphical for every $u \in \S^{n-1}$ (by taking $\Omega_1 = \interior P_{u^{\perp}} K$). 
An identical argument to the one above verifies that Definition \ref{def:SutK-general} of $S_u^t K$ coincides with the classical definition  (\ref{eq:SteinerForConvex}) of continuous Steiner symmetrization of a convex body for every $u \in \S^{n-1}$ and $t \in [0,1]$. 
\end{rem}

We will also require the following uniform estimates.
\begin{lem} \label{lem:uniform-ratio}
Let $K$ be a Lipschitz star body in $\R^n$. There exists a constant $M > 0$ such that for all $u \in \S^{n-1}$ for which $K$ is $u$-multi-graphical, for all $t \in [0,1]$ and for all $\theta \in \S^{n-1}$,
\[
\frac{1}{1 + M t} \leq \frac{\rho_{S^t_u K}(\theta)}{\rho_{K}(\theta)} \leq 1 + M t . 
\]
\end{lem}
\begin{proof}
Let $R > 0$ be such that $K \subseteq B_n(R)$, and let $L>0$ be the constant from Proposition~\ref{prop:StRemainsLip}, ensuring that $\Lip_{S_u^t K} \leq L$ for all $t \in [0,1]$. By Lemma \ref{lem:StLip}, $\mathring S^t_u K \subseteq \mathring S^0_u K + [-Rt,Rt] u$ and $\mathring S^0_u K \subseteq \mathring S^t_u K + [-Rt,Rt] u$. Taking the closure and using that $S_u^0 K = K$ by Corollary \ref{cor:Su01}, we deduce in particular that $S^t_u K \subseteq K + B_n(Rt)$ and $K \subseteq S^t_u K + B_n(Rt)$. Applying Lemma \ref{lem:adding-ball}, the assertion follows with $M = R L$. 
\end{proof}

\section{Admissible radial perturbations} \label{sec:admissible}

\begin{defn}[Admissible radial perturbation] \label{def:admissible}
Let $K$ be a star body in $\R^n$. A family of star-shaped sets $\{K_t\}_{t \in [0,1]}$ is called an admissible radial perturbation if $K_0 = K$ and $\{ [0,1] \ni t \mapsto \rho_{K_t}(\theta) \}_{\theta \in \S^{n-1}}$ are a.e.~equi-differentiable at $t=0^+$ in the following sense:
\begin{enumerate}
\item For almost every $\theta \in \S^{n-1}$, the following limit exists:
\begin{equation} \label{eq:arp-1}
 \dht{\rho_{K_t}(\theta)}  := \lim_{t \rightarrow 0^+} \frac{\rho_{K_t}(\theta) - \rho_{K_0}(\theta)}{t} . 
 \end{equation}
\item There exists $M > 0$ and $t_0 \in (0,1]$ such that for almost every $\theta \in \S^{n-1}$,
\begin{equation} \label{eq:arp-2}
\sup_{t \in (0,t_0]} \frac{|\rho_{K_t}(\theta) - \rho_{K_0}(\theta)|}{t} \leq M . 
\end{equation}
\end{enumerate}
\end{defn}

\begin{prop}[Stationary points] \label{prop:stationary}
Let $\{K_t\}_{t \in [0,1]}$ be an admissible radial perturbation of a star body $K$ in $\R^n$. Then, denoting $f(\theta) := \dt{\rho_{K_t}(\theta)}$, the following derivatives exist and are given by
\begin{align*}
\dht{|K_t|} & = \int_{\S^{n-1}} \rho_{K}^{n-1}(\theta) f(\theta) \, d\theta , \\
\dht{|I(K_t)|} & = (n-1) \int_{\S^{n-1}} \frac{\rho_{I^2 K}(\theta)}{\rho_K(\theta)} \rho_{K}^{n-1}(\theta) f(\theta)  \, d\theta .  
 \end{align*}
 Consequently, $I^2 K = c K$ if and only if $K$ is a stationary point for the functional
 \[
 \F_c(K) := |I(K)| - (n-1) c |K|,
 \]
 meaning that $\dt{\F_c(K_t)}  = 0$ for any admissible radial perturbation $\{K_t\}_{t \in [0,1]}$. 
 
 \noindent
 In particular, if $I^2 K = c K$ and $\dt{|K_t|} = 0$ then $\dt{|I(K_t)|}=0$. 
\end{prop}
\begin{proof}
Note that $f(\theta) = \dt{\rho_{K_t}(\theta)}$ exists  for a.e.~$\theta \in \S^{n-1}$ (and is thus Lebesgue measurable) by (\ref{eq:arp-1}), and is a bounded function on $\S^{n-1}$ by (\ref{eq:arp-2}); in particular, $f \in L^2(\S^{n-1})$.

If $K \subseteq B_n(R)$, (\ref{eq:arp-2}) implies in particular that for a.e.~$\theta \in \S^{n-1}$, $\sup_{t \in (0,t_0]} \rho_{K_t}(\theta) \leq R + M$, and so invoking (\ref{eq:arp-2}) again, we see that for all $m \geq 1$ and a.e.~$\theta \in \S^{n-1}$,
\[
\sup_{t \in (0,t_0]} \frac{\rho_{K_t}^m(\theta) - \rho_{K_0}^m(\theta)}{t} \leq C_{R,M,m} ,
\]
for some constant $C_{R,M,m} > 0$. By Lebesgue's Dominant Convergence Theorem, we may therefore exchange limit and integration:
\[
\dht{|K_t|} = \frac{d}{dt} \brac{\frac{1}{n} \int_{\S^{n-1}} \rho_{K_t}^n(\theta) \, d\theta }=  \int_{\S^{n-1}} \rho_{K}^{n-1}(\theta) f(\theta) \, d\theta .
\]
Similarly,
\begin{align*}
\dht{|IK_t|} & = \frac{d}{dt} \brac{\frac{1}{n} \int_{\S^{n-1}} |K_t \cap u^{\perp}|^n \, du} \\
& = \frac{d}{dt} \brac{\frac{1}{n} \int_{\S^{n-1}} \brac{ \frac{1}{n-1} \int_{\S^{n-1} \cap u^{\perp}} \rho^{n-1}_{K_t}(\theta) d\theta}^n \, du} \\
& = \int_{\S^{n-1}} |K \cap u^{\perp}|^{n-1} \int_{\S^{n-1} \cap u^{\perp}} \rho_K^{n-2}(\theta) f(\theta) \, d\theta \, du . 
\end{align*}
As $f  \in L^2(\S^{n-1})$, we proceed by (\ref{eq:Rad-symmetric}) as follows:
\begin{align*}
 & = \sigma_{n-1} \int_{\S^{n-1}} \rho^{n-1}_{IK}(u) \Rad(\rho_K^{n-2} f)(u) \, du \\
 & = \sigma_{n-1}  \int_{\S^{n-1}} \Rad(\rho^{n-1}_{IK})(u)  \rho_K^{n-2}(u) f(u) \, du \\
 & = (n-1) \int_{\S^{n-1}} \rho_{I^2 K}(u) \rho_K^{n-2}(u) f(u) \, du . 
\end{align*}
It follows that
\[
\dht{\F_c(K_t)} = (n-1) \int_{\S^{n-1}} \brac{\frac{\rho_{I^2 K}(\theta)}{\rho_K(\theta)} - c} \rho_{K}^{n-1}(\theta) f(\theta)  \, d\theta ,
\]
implying that $\dt{\F_c(K_t)} = 0$ if $I^2 K = c K$. Conversely, if $I^2 K \neq c K$, we can find a continuous $f : \S^{n-1} \rightarrow \R$ so that the right-hand side is non-zero; defining the star bodies $\{K_t\}_{t \in [0,1]}$ via $\rho_{K_t} = \rho_K + \eps t f$ for an appropriately small $\eps > 0$ yields an admissible radial perturbation for which $\dt{\F_c(K_t)} \neq 0$. This concludes the proof. 
\end{proof}

\begin{prop}[Continuous Steiner symmetrization is admissible for Lipschitz star bodies] \label{prop:admissible}
Let $K$ be a Lipschitz star body in $\R^n$, and let $u \in \U$ where $\U \subseteq \S^{n-1}$ is given by Theorem~\ref{thm:Lip-multi-graphical}. Then the continuous Steiner symmetrization $\{K_t := S_u^t K\}_{t \in [0,1]}$ is an admissible radial perturbation of $K$. 
\end{prop}
We remark that for convex bodies containing the origin in their interior, this was shown for all $u \in \S^{n-1}$ by Saroglou \cite[Section 4]{Saroglou-GeneralizedBCD}, but our setup is very different. 
\begin{proof}
The uniform estimates (\ref{eq:arp-2}) follow directly from Lemma \ref{lem:uniform-ratio} (and the fact that $B_n(r) \subseteq K \subseteq B_n(R)$). 
To establish (\ref{eq:arp-1}), we argue as follows. Denote $\Omega^*_\infty := \cup_m \Omega^*_m$, where $\Omega^*_m$ is the Lebesgue measurable subset of  $\Omega_m$ of full $\H^{n-1}$-measure where $\{f_i,g_i\}_{i=1,\ldots,m}$ are differentiable. Since $\H^{n-1}(\Omega_\infty \setminus \Omega^*_\infty) = 0$, since $\H^0(\partial K \cap L_u^y) < \infty$ for all $y \in \Omega_\infty$ and since $\H^{n-1}(\partial K \setminus (\Omega_\infty \times L_u)) = 0$, $\sigma$-sub-additivity implies $\H^{n-1}(\partial K \setminus (\Omega^*_\infty \times L_u)) = 0$. 
Since $\S^{n-1} \ni \theta \mapsto \rho_K(\theta) \theta \in \partial K$ is clearly a bi-Lipschitz map, it maps back and forth between $\H^{n-1}$-null-sets. 
It is therefore enough to show that (\ref{eq:arp-1}) holds for all $\theta \in \S^{n-1}$ such that $P_{u^{\perp}} \rho_K(\theta) \theta \in \Omega^*_\infty$. 

So let $\theta_0 \in \S^{n-1}$ be such that $\rho_K(\theta_0) \theta_0 = x_0 = y_0 + s_0 u$ with $y_0 \in \Omega^*_m$. By Remark \ref{rem:mgs}, as $x_0 \in \partial K$, we have $s_0 = h(y_0)$ for some $h \in \{ f_i,g_i \}_{i=1,\ldots,m}$, where the functions $f_1 < g_1 < \cdots < f_m < g_m$ are continuous in $\Omega_m$ and differentiable at $y_0$. Recall that $S_u^t K \cap (\Omega_m \times L_u) = \mathring S_u^t K \cap (\Omega_m \times L_u)$ by Corollary \ref{cor:closure}. By continuity, it follows that there exists $\eta > 0$ and $t_0 \in (0,1]$ such that defining $B = \interior B_{u^{\perp}}(y_0,\eta) \subset \Omega_m$, we have for all $y \in B$ and $t \in [0,t_0)$,
\[
\partial S_u^t K \cap L^y_u = y + u \set{ f_i(y) - \frac{f_i(y) + g_i(y)}{2} t , g_i(y) - \frac{f_i(y) + g_i(y)}{2} t  }_{i=1,\ldots,m} .
\]
Without loss of generality, we assume that $h = f_i$, and define
\[
\Phi : (B \times L_u)  \times \R \rightarrow \R ~,~ \Phi(y + s u , t) := s - \brac{ f_i(y) - \frac{f_i(y) + g_i(y)}{2} t } . 
\]
The function $\Phi$ is continuous on its domain and differentiable at $(x_0,0)$. Again, by continuity, we may choose $\delta > 0$ so that for all $x \in B \times (s_0-\delta,s_0+\delta)$ and $t \in [0,t_0)$,
\[
x \in \partial S_u^t K \;\; \Leftrightarrow \;\; \Phi(x,t) = 0  .
\]
Denoting $r_0 := \rho_K(\theta_0)$ and
\[
\varphi(r,t) := \Phi(r \theta_0 , t) = \frac{r}{r_0} s_0 - \brac{ f_i\brac{\frac{r}{r_0} y_0} - \frac{f_i(\frac{r}{r_0} y_0) + g_i(\frac{r}{r_0} y_0)}{2} t } ,
\]
we conclude that $\varphi(r,t)$ is differentiable at $(r_0,0)$ and continuous in a neighborhood thereof, and that for all $t \in [0,t_0)$,
\begin{equation} \label{eq:boundary}
\rho_{S_u^t K}(\theta_0) = r \;\; \Leftrightarrow \;\; r \theta_0 \in \partial S_u^t K \;\; \Leftrightarrow \;\; \varphi(r,t) = 0  
\end{equation}
(the first equivalence is due to the fact that $S_u^t K$ remain star bodies by Proposition \ref{prop:StRemainsLip}). 

By Proposition \ref{prop:Lip}, $K$ is star-shaped with respect to $B_n(\eps)$ for some $\eps > 0$, and so it contains the convex hull of $\{x_0\}$ and $B_n(\eps)$. This means that $\partial K$ must meet the ray $\R_+ \theta_0$ transversally at $x_0$ --- specifically, it is easy to check that
\[
\abs{\left . \frac{\partial \varphi(r,t)}{\partial r} \right |_{(r,t) = (r_0,0)}} = \frac{1}{r_0} \abs{s_0 - \scalar{\nabla f_i(y_0) , y_0} } \geq \frac{\eps}{r_0} > 0 .
\]
Consequently, by a version of the implicit function theorem for continuous functions which are differentiable at a given point \cite[Theorem E]{Halkin-ImplicitFunctionTheorems}, it follows that there exists $t_1 > 0$ and a continuous $\rho_0(t) : (-t_1,t_1) \rightarrow \R$,  differentiable at $t=0$, such that $\rho_0(0) = r_0 = \rho_K(\theta_0)$ and for $t \in (-t_1,t_1)$,
\[
\varphi(\rho_0(t),t) = 0  . 
\]
It follows by (\ref{eq:boundary}) that $\rho_{S_u^t K}(\theta_0) = \rho_0(t)$ for all $t \in [0,\min(t_0,t_1))$, and we conclude that $\dt{\rho_{S_u^t K}(\theta_0)} = \rho_0'(0)$ exists. This concludes the proof. 
\end{proof}

\begin{cor} \label{cor:derivative-exists}
Let $K$ be a Lipschitz star body in $\R^n$. Then, there exists a Lebesgue measurable $\U \subseteq \S^{n-1}$ of full measure (given by Theorem \ref{thm:Lip-multi-graphical}) so that for all $u \in \U$, $K$ is $u$-multi-graphical, $\dt{|I (S_u^t K)|}$ exists, and if $I^2 K = c K$ then $\dt{|I (S_u^t K)|}=0$.
\end{cor}
\begin{proof}
By Theorem  \ref{thm:Lip-multi-graphical} and the previous two propositions, for all $u \in \U$, $K$ is $u$-multi-graphical, $\{S_u^t K\}_{t \in [0,1]}$ is an admissible perturbation of $K$, and the derivative $\dt{|I(S_u^t K)|}$ exists. By Corollary \ref{cor:StVolume}, $|S_u^t K| = |K|$ for all $t \in [0,1]$, and hence $\dt{|S_u^t K|} = 0$. Consequently, if $I^2 K = c K$ then $\dt{|I(S_u^t K)|}=0$ by Proposition \ref{prop:stationary}.
\end{proof}

In the next section, we will see moreover that $\dt{|I (S_u^t K)|} \geq 0$, and characterize the equality conditions. 

\section{Characterization of equality under Steiner symmetrization} \label{sec:equality}

Let $K$ be a $u$-multi-graphical compact set in $\R^n$, let $\{\Omega_m\}_m$ and $\Omega_\infty$ be the open subsets of $P_{u^{\perp}} K$ from Definition \ref{def:mgs}, and recall that $\H^{n-1}(P_{u^{\perp}} K \setminus \Omega_\infty) = 0$. By Proposition \ref{prop:independent} and Corollary \ref{cor:u-finite}, $S_u^t K$ is well defined and $u$-finite for all $t \in [0,1]$. We also recall definition (\ref{eq:Iu}) of the functional $\I_u$, which when applied to $S_u^t K$ becomes
\begin{equation} \label{eq:Iu-again}
\I_u(S_u^t K) = \frac{2}{n}\int_{\Omega_\infty^n} \Delta(\tilde y_1,\ldots,\tilde y_{n-1})^{-1} |R_\y(S_u^t K) \cap \theta_\y^{\perp}|_{n-1} \, dy_1 \ldots dy_n ,
\end{equation}
where $\y = (y_1,\ldots,y_n) \in \Omega_\infty^n$, $\theta_\y \in \S^{n-1}$ denotes a linear dependency satisfying $\sum_{i=1}^n \theta_\y^i y_i = 0$ (uniquely defined up to sign on the subset of full measure where $\y$ is affinely independent), $(\tilde y_1,\ldots,\tilde y_{n-1})$ are explicit but presently irrelevant functions of $\y$,
and by Corollary \ref{cor:closure},
\begin{align*}
R_\y(S_u^t K) &= \{ (s^1,\ldots,s^n)  \in \R^n :  y_i + s^i u \in S_u^t K \, , \, i=1,\ldots,n\} \\
&= \{ (s^1,\ldots,s^n)  \in \R^n : y_i + s^i u \in \mathring S_u^t K \, , \, i=1,\ldots,n\}\\
& =  S^t (K \cap L_u^{y_1}) \times \cdots \times S^t (K \cap L_u^{y_n}) 
\end{align*}
is a finite disjoint union of rectangles in $\R^n$. Here and below, a rectangle always refers to a compact axis-aligned rectangle with non-empty interior,
and we identify $K \cap L_u^y$ with a subset of $\R$.

\subsection{Rectangles}

Let
\[
\dtl{\I_u(S_u^t K)} := \liminf_{t \rightarrow 0^+} \frac{\I_u(S_u^t K) - \I_u(S_u^0 K)}{t} . 
\]

\begin{lemma} \label{lem:rectangles-increasing}
Let $K$ be a $u$-multi-graphical compact set in $\R^n$. Then for a.e.~$\y \in \Omega_\infty^n$, $[0,1] \ni t \mapsto |R_\y(S_u^t K) \cap \theta_\y^{\perp}|_{n-1}$ is non-decreasing. In particular, $[0,1] \ni t\mapsto \I_u(S_u^t K)$ is non-decreasing and $\dtll{\I_u(S_u^t K)} \geq 0$. \end{lemma}
\begin{proof}
We may assume that $\y = (y_1,\ldots,y_n)$ are affinely independent so that $\theta_\y$ is uniquely defined (up to sign), and in addition exclude the case that $\theta_\y \in \{\pm e_i\}_{i=1,\ldots,n}$, as this corresponds to the null-set $\{ \y \in \Omega_\infty^n : \exists i =1,\ldots,n \;\; y_i = 0 \}$. 

For each $t \in [0,1]$, $R_t := R_\y(S_u^t K)$ is the disjoint union of finitely many rectangles $R_t^k$; we denote their centers by $c(R_t^k)$.  Let $0 = \tau_0 < \tau_1 < \cdots < \tau_N = 1$ denote the collision times of the $R_t^k$'s as they evolve in time. We will verify the monotonicity of $|R_t \cap \theta_\y^{\perp}|_{n-1}$ on $t \in [\tau_j,\tau_{j+1}]$ for each $j$. 
For $t \in [\tau_j, \tau_{j+1})$, each $R_{\tau_j}^k$ evolves independently as $R_t^k = R_{\tau_j}^k - \frac{t-\tau_j}{1-\tau_j} c(R_{\tau_j}^k)$. 
Therefore, for all $t \in [\tau_j, \tau_{j+1}]$:
\begin{equation} \label{eq:rectangle-sum}
|R_t \cap \theta_\y^{\perp}|_{n-1} = \sum_{k} \abs{\brac{R_{\tau_j}^k - \frac{t-\tau_j}{1-\tau_j} c(R_{\tau_j}^k)} \cap \theta_\y^{\perp}}_{n-1} ;
 \end{equation}
 this is trivial for $t \in [\tau_j,\tau_{j+1})$ since the rectangles on the right are disjoint, but also holds at the collision time $t = \tau_{j+1}$ since $|\partial R \cap \theta_y^{\perp}|_{n-1} = 0$ for any rectangle $R$, as $\theta_\y \notin \{\pm e_i\}_{i=1,\ldots,n}$. This reduces our task to showing that each of the summands on the right in (\ref{eq:rectangle-sum}) is non-decreasing in $t \in [\tau_j,\tau_{j+1}]$, which is a consequence of the next lemma. 
 \end{proof}

 \begin{lemma} \label{lem:rectangle-equality}
 Let $R = \Pi_{i=1}^n [c^i-\h^i ,c^i + \h^i]$ ($\h^i > 0$) denote a rectangle in $\R^n$ centered at $c = (c^i)$, and let $\theta \in \S^{n-1} \setminus \{\pm e_i\}_{i=1,\ldots,n}$. Then $[0,1] \ni t \mapsto f(t) := |(R - t c) \cap \theta^{\perp}|_{n-1}$ is non-decreasing, the right-derivative $\dt{f(t)} \geq 0$ exists, and it is equal to $0$ if and only if either
 \begin{enumerate}
 \item $|R \cap \theta^{\perp}|_{n-2} = 0$, or 
 \item $c \in \theta^{\perp}$, or 
   \item \label{it:rectangle-equality-3} 
 $\theta^{\perp}$ intersects exactly $n-1$ pairs of opposing facets of $R$ (and possibly an additional single facet, but not its interior).
  \end{enumerate}
A useful necessary condition for having $\dt{f(t)} = 0$ is obtained by replacing (\ref{it:rectangle-equality-3}) with
\begin{enumerate}
\item[(3')] $\theta^{\perp}$ intersects exactly $n-1$ pairs of opposing facets of the centered rectangle $R - c$ (and no other facets). 
\end{enumerate}
 \end{lemma}
 
  Figure~\ref{fig:Rmove} provides an illustration of the translation which a rectangle undergoes in Lemma~\ref{lem:rectangle-equality} as it is being centered. Figure~\ref{fig:3cases} illustrates the three equality cases described in the lemma. 
 
 \begin{figure} \centering
\includegraphics[width=7cm]{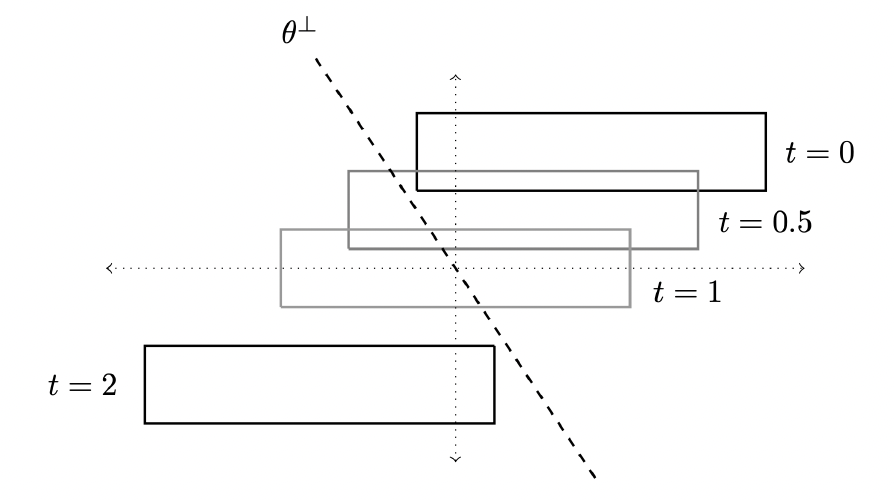}
\caption{An illustration of a rectangle being translated towards the origin in Lemma~\ref{lem:rectangle-equality}.}
\label{fig:Rmove}
\end{figure}

  \begin{figure} \centering
\includegraphics[width=6cm]{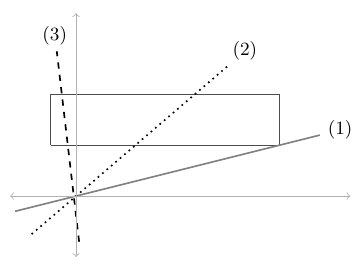}
\caption{An illustration of the three cases of equality in Lemma~\ref{lem:rectangle-equality}}.\label{fig:3cases}
\end{figure}

 \begin{proof}[Proof of Lemma \ref{lem:rectangle-equality}]
 Denote by $R_0 = R - c$ the centered rectangle, and define the function $g : \R \rightarrow \R_+$ by
\[
 g(s)^{n-1} := f(1+s) = |R_0 \cap (s c + \theta^{\perp})|_{n-1} = |R_0 \cap (s \scalar{c,\theta} \theta + \theta^{\perp})|_{n-1} .
\]
If $\scalar{c,\theta} = 0$ then $f$ is constant and there is nothing to prove, so we may exclude this case (as one of the cases of equality). 
Since $R_0$ is compact with non-empty interior, $g$ is continuous on its support $[-M,M]$ and $M \in (0,\infty)$. 
Moreover, $g$ is even and concave on its support by Brunn's concavity principle (\ref{eq:Brunn}). 
Consequently, $g$ is non-decreasing on $[-1,0]$ and hence $f$ is non-decreasing on $[0,1]$, yielding the first part of the claim. 

Now, if $R \cap \theta^{\perp} = \emptyset$ (equivalently, $M < 1$) then trivially $\dt{f(t)} = 0$. 
If $M=1$, since we assumed that $\theta \notin \{\pm e_i\}_{i=1,\ldots,n}$, necessarily $\theta^{\perp}$ intersects a face of $R$ of dimension $k \leq n-2$. In that case, for some $\eps > 0$ and all $t \in [0,\eps]$, $R \cap (c t + \theta^{\perp})$ is congruent to $R_k \times t \Delta_{n-k-1}$, where $R_k$ is a $k$-dimensional rectangle and $\Delta_{n-k-1}$ is an $(n-k-1)$-dimensional simplex. 
  Consequently, $f(t) = a t^{n-k-1}$ for some $a > 0$ and all $t \in [0,\eps]$, and so $f$ is differentiable from the right at $t=0$, and $\dt{f(t)} = 0$ iff $k < n-2$. Note that we may combine both of the prior two scenarios into the single statement that ``$M \leq 1$ and $\dt{f(t)} = 0$" iff $|R \cap \theta^{\perp}|_{n-2} = 0$. 

If $M > 1$, since $g$ is differentiable from the right on $(-M,M)$, 
$f$ is differentiable from the right at $t=0$, and as $g(-1) > 0$, $\dt{f(t)} =0$ iff $\dsa{g(s)}{-1}=0$ iff $g$ is constant on $[-1,1]$ (since by concavity, the non-negative right-derivative $\dsa{g(s)}{s_0}$ is non-increasing for $s_0 \in [-1,0)$). It follows by the equality conditions of the Brunn-Minkowski inequality that $R_0 \cap (s \scalar{c,\theta} \theta + \theta^{\perp})$ for $s \in [-1,1]$ are all translates of the central section $R_0 \cap \theta^{\perp}$, i.e.~coincide with $(R_0 \cap \theta^{\perp}) + s T$ for some $T \in \R^n$ such that $\scalar{T,\theta} = \scalar{c,\theta}$.
The central section is an origin-symmetric $(n-1)$-dimensional convex body, and hence $\theta^{\perp}$ must intersect $m \geq n-1$ pairs of opposing facets of $R_0$ (otherwise it would not be bounded). If $\theta^{\perp}$ intersects the pair of facets perpendicular to $e_i$, then necessarily the translation direction $T$ must satisfy $\scalar{T,e_i} = 0$ (otherwise $(R_0 \cap \theta^{\perp}) + s T$ would not lie inside $R_0$ for $s \neq 0$). So if $m=n$ this means $T=0$ and hence $\scalar{c,\theta}=\scalar{T,\theta} = 0$, but this case was already excluded. Consequently, if $\scalar{c,\theta}\neq 0$, $M > 1$ and $\dt{f(t)} =0$ then $\theta^{\perp}$ intersects exactly $n-1$ pairs of opposing facets of $R_0$ (and no other facets). The translated sections $(R_0 \cap \theta^{\perp}) + s T$ of $R_0$ will still intersect the same $n-1$ pairs of opposing facets of $R_0$ for all $s \in [-1,1]$, and no other facets if $s \in (-1,1)$; at times $s \in \{-1,1\}$, these sections may intersect an additional single facet but not its interior. 
Translating everything back to a statement regarding $R$ at time $t=0$, it follows that $\theta^{\perp}$ intersects exactly $n-1$ pairs of opposing facets of $R$, and possibly an additional single facet but not its interior. 

Conversely, assume that the latter scenario occurs  (a direction which we do not need in the sequel, but nevertheless establish for completeness). Then, there exists $i=1,\ldots,n$ so that $s c + \theta^{\perp}$ for $s=-1$ does not intersect $\interior F_i$, the union of the (relative) interiors of the pair of facets of $R_0$ perpendicular to $e_i$, and by symmetry also for $s=1$. Denoting the polytope $P_0 := R_0 \cap ([-1,1] c + \theta^{\perp})$, it follows, since $\partial P_0$ is connected and is a subset of $\partial R_0 \cup (R_0 \cap (\{-1,1\} c + \theta^{\perp}))$, that either $\partial P_0$ contains $\interior F_i$ or is disjoint from it. Since $P_0$ is closed and convex, the former possibility would imply that $P_0 = R_0$ which is impossible, since this would mean that either $R \cap \theta^{\perp} = \emptyset$ or (as $\theta \neq \pm e_i$) that $R \cap \theta^{\perp}$ is a face of $R$ of dimension at most $n-2$, and so in either case $\theta^{\perp}$ cannot intersect $n-1$ opposing pairs of facets of $R$. Consequently $\partial P_0 \cap \interior F_i = \emptyset$, and since $P_0 \subseteq \closure \conv(\interior F_i)$, we deduce that $P_0 \cap \interior F_i = \emptyset$. This means that for all $s \in [-1,1]$, $R_0 \cap (s c + \theta^{\perp})$ coincides with $C_0 \cap (s c + \theta^{\perp})$, where $C_0$ is the cylinder $\sum_{j \neq i} [-\h^j,\h^j] e_j + \R e_i$. This implies that these sections are all translates of each other, and hence the function $f$ is constant on $[0,1]$ and $\dt{f(t)} =0$. 
  \end{proof}

We can now immediately deduce the following. 

\begin{proposition} \label{prop:rectangle-equality}
Let $K$ be a $u$-multi-graphical compact set in $\R^n$. If $\dtll{\I_u(S_u^t K)} =  0$
then for a.e.~$\y = (y_1,\ldots,y_n) \in (P_{u^{\perp}} K)^n$, $R_\y$ consists of a finite disjoint union of rectangles $\{R_\y^k\}$ in $\R^n$, 
such that for each rectangle $R_\y^k$, either
\begin{enumerate} 
\item \label{it:equality-cond1} $\theta_\y^{\perp}$ essentially does not intersect $R_\y^k$, i.e., $|R_\y^k \cap \theta_\y^{\perp}|_{n-2} = 0$, or 
\item \label{it:equality-cond2} $\theta_\y^{\perp}$ passes through the center $c(R_\y^k)$ of $R_\y^k$, i.e., $\scalar{\theta_\y , c(R_\y^k)} = 0$, or 
\item \label{it:equality-cond3} $\theta_\y^{\perp}$ intersects exactly $n-1$ pairs of opposing facets of the centered rectangle $R_\y^k -  c(R_\y^k)$ (and no other facets). 
\end{enumerate}
\end{proposition}
\begin{proof}
Recalling (\ref{eq:Iu-again}) and Lemma \ref{lem:rectangles-increasing}, we may appeal to Fatou's lemma to bound $\dtll{\I_u(S_u^t K)}$ from below. Consequently, if $\dtll{\I_u(S_u^t K)} = 0$, it follows that $ \dtll{|R_\y(S_u^t K) \cap \theta_\y^{\perp}|_{n-1}} = 0$ for a.e.~$\y \in \Omega_\infty^n$ and hence for a.e.~$\y \in (P_{u^{\perp}} K)^n$. The representation (\ref{eq:rectangle-sum}) for $\tau_0=0$ and Lemma \ref{lem:rectangle-equality} conclude the proof. 
\end{proof}

We can now provide a complete proof of Theorem \ref{thm:intro-dIdt} from the Introduction.
\begin{proof}[Proof of Theorem \ref{thm:intro-dIdt}]
If $K$ is a Lipschitz star body in $\R^n$, then by Corollary \ref{cor:derivative-exists}, for all $u \in \U$, $K$ is $u$-multi-graphical and the derivative $\dt{|I(S_u^t K)|}$ exists. By Proposition \ref{prop:StRemainsLip} and Corollary \ref{cor:u-finite}, we also know that for such $u$'s, $S_u^t K$ remains a Lipschitz star body (and hence with radially negligible boundary) and $u$-finite, and so by Theorems \ref{thm:I0} and \ref{thm:Iu}, $|I(S_u^t K)| = I_0(S_u^t K) = I_u(S_u^t K)$ for all $t \in [0,1]$. Lemma \ref{lem:rectangles-increasing} therefore implies that this function is non-decreasing and that $\dt{|I(S_u^t K)|} = \dt{\I_u(S_u^t K)}\geq 0$, and if equality occurs then the conclusion of Proposition \ref{prop:rectangle-equality} must hold.  
\end{proof}

\subsection{Intersecting all facets of a rectangle}

To better understand condition (\ref{it:equality-cond3}) from Proposition \ref{prop:rectangle-equality}, we have the following lemma. 

\begin{lemma} \label{lem:all-facets}
Let $R = \Pi_{i=1}^n [c^i - \h^i , c^i + \h^i]$ ($\h^i > 0$) be a rectangle in $\R^n$, and denote by $F_i$ the union of its two facets perpendicular to $e_i$ ($i=1,\ldots,n$). Let $\theta \in \R^n$, and assume that $R \cap \theta^{\perp} \neq \emptyset$. Then $ F_i \cap \theta^{\perp} \neq \emptyset$ for all $i=1,\ldots,n$ if and only if
\begin{equation} \label{eq:max-R0}
 \max_{x \in R} \abs{\scalar{\theta,x}} \geq 2 \max_{i=1,\ldots,n} |\theta_i| \h^i  .
 \end{equation}
 Equivalently, denoting $B^n_1(\h) = \conv\{\pm \h^i e_i\}_{i=1,\ldots,n}$, $\theta^{\perp}$ does not intersect all of the $F_i$'s if and only if
 \begin{equation} \label{eq:max-R}
 P_{\sspan \theta} R  \subset \interior(2 P_{\sspan \theta} B^n_1(\h)) . 
 \end{equation}
 \end{lemma}
 \begin{proof}
 Since we are given that $R \cap \theta^{\perp} \neq \emptyset$, we know that
 \begin{equation} \label{eq:zeroin}
 0 \in \sum_{i=1}^{n} \theta_i [c^i - \h^i, c^i + \h^i] . 
 \end{equation}
 Without loss of generality, we check the intersection of $\theta^{\perp}$ with $F_n$. 
 Note that
 \[
\theta^{\perp} \cap F_n = \emptyset \;\; \Leftrightarrow \;\; 0 \notin \sum_{i=1}^{n-1} \theta_i [c^i - \h^i, c^i + \h^i] + \theta_n \{ c^n - \h^n , c^n + \h^n \} .
 \]
 The right-hand side is the union of two intervals whose convex hull contains the origin by (\ref{eq:zeroin}). 
Consequently, we may proceed as follows:
 \begin{align*}
  \Leftrightarrow \; & \; \;  - \abs{\theta_n} \h^n  < \sum_{i=1}^{n-1} \theta_i [c^i - \h^i, c^i + \h^i] + \theta_n c^n < +  \abs{\theta_n} \h^n \\
 \Leftrightarrow \; & \; \;  - 2  \abs{\theta_n} \h^n  < \sum_{i=1}^{n-1} \theta_i [c^i - \h^i, c^i + \h^i] + \theta_n c^n + \theta_n [-\h^n, \h^n]  < +2  \abs{\theta_n} \h^n
 \\
 \Leftrightarrow \; & \; \;  \max_{x \in R} \abs{\scalar{\theta,x}} < 2 \abs{\theta_n} \h^n . 
 \end{align*}
Here we used $a < I < b$ to signify that $a < \min I \leq \max I < b$. 
Replacing the $n$-th coordinate with an arbitrary one, (\ref{eq:max-R0}) follows.  Since
 the linear functional $\scalar{\theta,\cdot}$ attains its maximum over $B^n_1(\h)$ on its vertices, the right-hand side of (\ref{eq:max-R0}) is equal to $2 \max_{x \in B^n_1(\h)} \scalar{\theta, x}$, and so the negation of (\ref{eq:max-R0}) is seen to be equivalent to (\ref{eq:max-R}). 
 \end{proof}
 
 Applying this to the centered rectangle $R^k_y - c(R^k_y) = B^n_\infty(\h^k_y)$, where $B^n_\infty(\h) = \Pi_{i=1}^n [-\h^i,\h^i]$, condition (\ref{it:equality-cond3}) of Proposition \ref{prop:rectangle-equality} implies that
\[ 
P_{\sspan \theta_\y} B^n_\infty(\h^k_y) \subset \interior (2 P_{\sspan \theta_\y} B^n_1(\h^k_\y)) .
\]

\subsection{Mid-points of fibers lie on a hyperplane}

We are now finally ready to utilize the crucial assumption that $n \geq 3$. 

\begin{thm} \label{thm:mid-points}
Let $K$ be a Lipschitz star body in $\R^n$ with $n \geq 3$, and let $\U \subseteq \S^{n-1}$ be the subset of full measure from Theorem \ref{thm:Lip-multi-graphical}.  There exists $\delta > 0$ so that if $\dt{|I(S_u^t K)|} = 0$ for some $u \in \U$, then for all $y \in B_{u^{\perp}}(\delta)$, the one-dimensional fibers $K \cap L_u^y$ are intervals whose mid-points lie on a common hyperplane through the origin.
\end{thm}

For the proof, we will require Lemma \ref{lem:intro-linear} from the Introduction, which we repeat here for the reader's convenience.

\begin{lemma} \label{lem:linear}
Let $f : B \rightarrow \R$ be a function on a centered open Euclidean ball $B \subseteq \R^{n-1}$, $n \geq 3$, and let $\Theta \subset \R^n$ be a non-empty open set. Assume that for all $\theta \in \Theta$,
for every affinely independent $y_1,\ldots,y_n \in B$ such that $\sum_{i=1}^n \theta_i y_i = 0$, it holds that $\sum_{i=1}^n \theta_i f(y_i) = 0$. 
Then, $f$ must be a linear function on $B \setminus \{0\}$.   
\end{lemma}
\begin{rem}
We do not assume that $f$ is continuous, hence the conclusion need only hold on the punctured ball. Indeed, if $\Theta$ does not intersect the coordinate axes, then we will never have access to $f(0)$ in our assumption, and so the value of $f$ at the origin can be arbitrary. In addition, note that without further assumptions on $\Theta$, the lemma is false for $n=2$, as $f$ may only be piecewise linear (separately on $(-\infty,0)$ and $(0,\infty)$). 
\end{rem}
\begin{proof}[Proof of Lemma \ref{lem:linear}]
Let $\theta \in \Theta$ be such that $\theta_i \neq 0$ for all $i=1,\ldots,n$ and $\sum_{i=1}^n \theta_i \neq 0$ (as $\Theta$ is an open set, this is always possible). Given linearly independent $y_1,\ldots,y_{n-1} \in \R^{n-1}$, define $y_n = -\frac{1}{\theta_n} \sum_{i=1}^{n-1} \theta_i y_i$, implying that $\sum_{i=1}^n \theta_i y_i = 0$, that $y_1,\ldots,y_n$ are affinely independent, and that moreover, the vectors $\{ y_i \}_{i \in I}$ for any $|I| = n-1$ are linearly independent. By relabeling indices if necessary, we may assume that $|y_n|= \max_{i=1,\ldots,n} |y_i|$.
 By simultaneously scaling and rotating all $y_i$'s, we may in fact ensure that $y_n$ is an arbitrary element of $B \setminus \{0\}$, and that all $y_i \in B \setminus \{0\}$. 
 
 Since $\Theta$ is an open set and since $y_1,\ldots,y_{n-1}$ are linearly independent (also after the relabeling of indices), we claim there exists an open neighborhood $N(y_n) \subset B \setminus \{0\}$ of $y_n$ such that for all $y_n' \in N(y_n)$, $y_1,\ldots,y_{n-1},y_n'$ are still affinely independent and there exists $\theta' \in \Theta$ such that $\sum_{i=1}^{n-1} \theta'_i y_i + \theta'_n y'_n = 0$. 
 Indeed, start with a neighborhood $N_0(y_n) \subset B \setminus \{0\}$ which is disjoint from the affine hull of $\{y_1,\ldots,y_{n-1}\}$, ensuring that $y_1,\ldots,y_{n-1},y_n'$ remain affinely independent for all $y_n' \in N_0(y_n)$. Now, denoting by $\Y$ the $(n-1) \times (n-1)$ the full-rank matrix whose columns are given by $y_1,\ldots,y_{n-1}$, we can choose $\theta'$ to satisfy the following linear system of equations:
\[
\Y \cdot P_{n-1} \theta' = -\theta'_n y'_n ~,~ \theta'_n = \theta_n ,
\]
where $P_{n-1} : \R^n \rightarrow \R^{n-1}$ denotes projection onto the first $n-1$ coordinates. Consequently, denoting $\delta y_n = y'_n - y_n$ and $\delta \theta = \theta' - \theta$, we have
\[
\Y \cdot P_{n-1} \delta \theta = -\theta_n \delta y_n . 
\]
Therefore, if $B_n(\theta,\eps) \subset \Theta$, we may take $N_1(y_n)$ to be the interior of the (non-degenerate) ellipsoid $y_n - \frac{1}{\theta_n} \Y B_{n-1}(0,\eps)$, and set $N(y_n) =N_0(y_n) \cap N_1(y_n)$.

 Our assumption implies that $\sum_{i=1}^{n-1} \theta'_i f(y_i) + \theta'_n f(y'_n) = 0$. Consider the $(n-1)$-dimensional linear subspace $H$ in $\R^{n-1} \times \R$ spanned by $\{(y_i,f(y_i))\}_{i=1,\ldots,n-1}$. It follows that $(y_n' , f(y_n'))$ must also lie on $H$ for all $y_n' \in N(y_n)$, and we conclude that $f$ is linear on $N(y_n)$. 
 
 We have shown that for an arbitrary point $y_n \in B \setminus \{0\}$ there exists an open neighborhood $N(y_n) \subset B \setminus \{0\}$ of $y_n$ such that $f$ coincides with a linear function $\ell_{y_n}$ on $N(y_n)$. To show that $f$ is linear on the entire $B \setminus \{0\}$, we need to show $\ell_{x_0} = \ell_{x_1}$ for all $x_0,x_1 \in B \setminus \{0\}$. Since $B \setminus \{0\}$ is connected when $n \geq 3$, we may connect $x_0,x_1$ using a (compact) path $P$. Extracting a finite open subcover of $P \subset \cup_{y_n \in P} N(y_n)$, and using that two linear functions defined on two overlapping open sets must coincide (because a linear function on a non-empty open set $\Omega \subset \R^{n-1}$ uniquely extends to the entire $\R^{n-1}$), it follows that $\ell_{x_0} = \ell_{x_1}$, concluding the proof. 
\end{proof}

\begin{proof}[Proof of Theorem \ref{thm:mid-points}]
Let $B_n(r) \subseteq K \subseteq B_n(R)$. 
By Proposition \ref{prop:graphical}, there exists $\delta >0 $ with $B_n(\delta) \subset \interior K$ such that $K$ is equi-graphical over $B_n(\delta)$. In particular, for all $u \in \S^{n-1}$ and $y \in B_{u^{\perp}}(\delta)$, $K \cap L_u^y$ is a closed interval $y + [f_u(y),g_u(y)] u$ with $f_u(y) < 0 < g_u(y)$; we denote its length by $2 \h_u(y) = g_u(y) - f_u(y)$ and center by $c_u(y) = \frac{f_u(y) + g_u(y)}{2}$. Since $g_u(y) = F(y,u)$ and $f_u(y) = -F(y,-u)$ for some uniformly continuous $F : B_n(\delta) \times \S^{n-1} \rightarrow \R$, the functions $\{f_u, g_u\}_{u \in \S^{n-1}}$ are equicontinuous (as their modulus of continuity is uniformly bounded from above by that of $F$). Consequently, by making $\delta > 0$ smaller if necessary, we can ensure that $\abs{\h_u(y) - \h_u(0)}\leq \eps r \leq \eps \h_u(0)$ for all $u \in \S^{n-1}$ and $y \in B_{u^{\perp}}(\delta)$. Here $\eps > 0$ is a fixed constant chosen such that $\frac{1+\eps}{1-\eps} < \frac{5}{4}$. 
Now fix $u \in \U$ and assume that $\dt{|I(S_u^t K)|}= 0$. By Theorem \ref{thm:intro-dIdt}, we know that for a.e.~$\y \in (P_{u^{\perp}} K)^n$ the conclusion of Proposition \ref{prop:rectangle-equality} holds for $R_\y= K \cap L_u^{y_1} \times \cdots \times  K \cap L_u^{y_n}$. Since when $\y \in B_{u^{\perp}}(\delta)^n$, $R_\y$ is a single rectangle $\Pi_{i=1}^n [c^i_\y - \h^i_\y , c^i_\y + \h^i_\y]$, we conclude by Lemma \ref{lem:all-facets} and the subsequent paragraph that for a.e.~$\y \in B_{u^{\perp}}(\delta)^n$, either $|R_\y \cap \theta_\y^{\perp}|_{n-2} = 0$, or else $\scalar{c_\y , \theta_\y} = 0$, or else $P_{\sspan \theta_\y} B^n_\infty(\h_\y) \subset \interior (2 P_{\sspan \theta_\y} B^n_1(\h_\y))$. The first scenario is impossible since $R_\y$ contains the origin in its interior, so we concentrate on the remaining two.

Let $\Theta = \{ \theta \in \R^n : \norm{\theta}_{1} > 2.5 \norm{\theta}_{\infty} \}$. Since $n \geq 3$, this is a non-empty open cone (this would not be the case if $n=2$ since $\norm{\theta}_{1} \leq n \norm{\theta}_{\infty}$). If $\theta_\y \in \Theta$, we claim that $|P_{\sspan \theta_\y} B^n_\infty(\h_\y) |_1 > 2 |P_{\sspan \theta_\y} B^n_1(\h_\y)|_1$, and so the third scenario is impossible. Indeed,
\begin{align*}
 |P_{\sspan \theta_\y} B^n_\infty(\h_\y)|_1 &  \geq (1-\eps) \h_u(0) |P_{\sspan \theta_\y} B_\infty^n|_1 = (1-\eps) \h_u(0) 2 \norm{\theta_\y}_{1} ,\\
|P_{\sspan \theta_\y} B^n_1(\h_\y)|_1 & \leq  (1+\eps) \h_u(0)  |P_{\sspan \theta_\y} B^n_1|_1 = (1+\eps) \h_u(0) 2 \norm{\theta_\y}_{\infty} ,
\end{align*}
and since $\norm{\theta_\y}_{1} > 2.5 \norm{\theta_\y}_{\infty}$ and $\frac{1-\eps}{1+\eps} > \frac{4}{5}$, the third scenario is disqualified as well.

Consequently, for a.e.~$\y = (y_1,\ldots,y_n) \in B_{u^\perp}(\delta)^n$ which are affinely independent with $\theta_\y \in \Theta$, necessarily $0 = \scalar{c_\y , \theta_\y} = \sum_{i=1}^n \theta_\y^i c_u(y_i)$. Since $c_\y = (c_u(y_1),\ldots,c_u(y_n))$ is continuous in $\y \in B_{u^\perp}(\delta)^n$, and so is $\theta_\y \in \S^{n-1}$ on the relatively open subset of affinely indepdendent vectors (as in the proof of Lemma \ref{lem:linear}), it follows that the statement in the previous sentence holds for all such $\y$'s, not just almost everywhere. Applying Lemma \ref{lem:linear} to the function $c_u(y)$, it follows that $c_u$ is a linear function on $\interior (B_{u^\perp}(\delta))\setminus \{0\}$; by continuity of $c_u$, this extends to the entire $B_{u^\perp}(\delta)$. In other words, all of the mid-points of fibers over $y \in B_{u^\perp}(\delta)$ lie on a common hyperplane through the origin, concluding the proof. 
\end{proof}

\subsection{Characterization of ellipsoids}

To finish the proof of Theorem \ref{thm:intro-main}, we need the following simple adaptation of Soltan's Theorem \ref{thm:intro-Soltan} from \cite{Soltan-EllipsoidViaMidpoints}. This is a local extension of the classical Bertrand--Brunn characterization of ellipsoids (see \cite[Section 8]{Soltan-EllipsoidsSurvey} for a historical discussion).
\begin{thm}[after Soltan \cite{Soltan-EllipsoidViaMidpoints}] \label{thm:Soltan}
Let $K$ denote a compact set in $\R^n$ which is equi-graphical over $B_n(\delta)$ for some $\delta > 0$. Assume that for a dense subset of $u$'s in $\S^{n-1}$, the mid-points of all segments of $K$ parallel to $u$ passing through $B_n(\delta)$  lie on a common hyperplane. Then $K$ is an ellipsoid. If these common hyperplanes all pass through the origin, then $K$ is a centered ellipsoid. 
\end{thm}
In particular, by Proposition \ref{prop:graphical}, this applies to any Lipschitz star body $K$. 
\begin{proof}
Soltan's proof of Theorem \ref{thm:intro-Soltan} (say, with $p=0$) in \cite[Section 7]{Soltan-EllipsoidViaMidpoints} does not invoke convexity beyond knowing that $K$ is $u$-graphical over $B_{u^{\perp}}(\delta)$ for all $u \in \S^{n-1}$. Consequently, to invoke Theorem \ref{thm:intro-Soltan}, it remains to show that the mid-point property holds not only for a dense subset of $u$'s in $\S^{n-1}$ but actually for every $u \in \S^{n-1}$. But since $K$ is assumed equi-graphical, the mid-points are given by $\frac{f_u(y) + g_u(y)}{2} = \frac{F(y,u) - F(y,-u)}{2}$ for some (uniformly) continuous function $F : B_n(\delta) \times \S^{n-1} \rightarrow \R$, and so this is immediate by continuity and the fact that affine functions (with bounded coefficients) are closed under pointwise convergence. We deduce that $K$ must be an ellipsoid, and if all of the mid-point hyperplanes pass through the origin, it must be centered. 
\end{proof}

\section{Tying everything together} \label{sec:proof}

We can now finally present the proof of Theorem \ref{thm:intro-main} and Corollary \ref{cor:intro-main}. For completeness, we also present a proof of the trivial directions. To this end, recall that for any star body $K$ in $\R^n$ and  a non-singular linear map $T \in \GL_n$ (see \cite[Theorem 8.1.6]{GardnerGeometricTomography2ndEd}),
\begin{equation} \label{eq:ITK}
I(T K) = |\det T| (T^{-1})^*(I K) .
\end{equation}
In particular,
\[
I(c K) = c^{n-1} I K \;\;\; \forall c > 0 ,
\]
and we see that $I^2$ is $\GL_n$-covariant in the following sense:
\[
I^2 (T K) = |\det T|^{n-2} \, T(I^2 K) .
\]
Clearly $I B_n = \omega_{n-1} B_n$, and $I^2 B_n = \omega_{n-1}^{n-1} I B_n = \omega_{n-1}^n B_n$.

\begin{proof}[Proof of Theorem \ref{thm:intro-main}] 
Let $K$ be a centered ellipsoid in $\R^n$ ($n \geq 2$), and write $K = T(B_n)$ for some $T \in \GL_n$. Then,
\[
I^2 K = |\det T|^{n-2} \, T(I^2 B_n) = \omega_{n-1}^n |\det T|^{n-2} K ,
\]
and we see that $I^2 K = c K$ for an appropriate $c > 0$.

Conversely, let $K$ be a star body in $\R^n$, $n \geq 3$, so that $I^2 K = c K$, $c > 0$. Recall that $\rho_{I K} = \omega_{n-1} \Rad(\rho_K^{n-1})$, where $\Rad$ denotes the spherical Radon (or Funk) transform, and hence
\begin{equation} \label{eq:RadRad}
c \rho_K = \rho_{I^2 K} = \omega_{n-1}^n \Rad(\Rad(\rho_K^{n-1})^{n-1}) .
\end{equation}
By Theorem \ref{thm:regularity} in the Appendix, since $n \geq 3$ it follows that $\rho_K$ is $C^\infty$ smooth, and in particular Lipschitz continuous, so $K$ is a Lipschitz star body. By Theorem \ref{thm:Lip-multi-graphical} there exists a Lebesgue measurable $\U \subseteq \S^{n-1}$ of full measure such that for all $u \in \U$, $K$ is $u$-multi-graphical (recall Definition \ref{def:mgs}). By the results of Subsection \ref{subsec:StK}, the continuous Steiner symmetrization $\{S_u^t K\}_{t \in [0,1]}$ is a well-defined family of compact sets satisfying $|S_u^t K| = |K|$ for all $t \in [0,1]$. Moreover,  $S_u^t K$ are (uniformly) Lipschitz star bodies with $S_u^0 K = K$ by Proposition \ref{prop:StRemainsLip} and Corollary \ref{cor:Su01}, and constitute an admissible radial perturbation of $K$ (recall Definition \ref{def:admissible}) by Proposition \ref{prop:admissible}. Since $I^2 K = c K$, it follows by Proposition \ref{prop:stationary} that $K$ is a stationary point for the functional $\F_c(K) = |I K| - (n-1) c |K|$, and since $\dt{|S_u^t K|}= 0$, we deduce in Corollary \ref{cor:derivative-exists} that $\dt{|I(S_u^t K)|}$ exists and is equal to $0$. Theorem \ref{thm:intro-dIdt} tells us that (for any Lipschitz star body $K$) $|I(S_u^t K)| = \I_0(S_u^t K) = \I_u(S_u^t K)$ is a non-decreasing function in $t \in [0,1]$, and provides some geometric information on $K$ whenever $\dt{|I(S_u^t K)|} = 0$. Crucially utilizing that $n \geq 3$, this is further refined in Theorem \ref{thm:mid-points}, stating that there exists $\delta > 0$ (independent of $u \in \U$) so that for all $y \in B_{u^{\perp}}(\delta)$, the one-dimensional fibers $K \cap L_u^y$ are intervals whose mid-points lie on a common hyperplane through the origin. As this holds for all $u \in \U$, applying Proposition \ref{prop:graphical} and Theorem \ref{thm:Soltan}, we conclude that $K$ must be a centered ellipsoid. 
\end{proof}

\begin{rem}
The same argument applies if $K$ is only assumed to be a star-shaped bounded Borel set in $\R^n$ ($n \geq 3$), and $I^2 K = c K$ is only assumed to hold up to an $\H^n$-null-set. These assumptions mean that $\rho_K$ is a non-negative function in $L^\infty(\S^{n-1})$, and that (\ref{eq:RadRad}) holds on $\S^{n-1}$ up to an $\H^{n-1}$-null-set (by integration in polar coordinates and Fubini), i.e.~as functions in $L^\infty(\S^{n-1})$. In that case, Theorem \ref{thm:regularity} in the Appendix shows that up to modifying $\rho_K$ on an $\H^{n-1}$-null-set (which amounts to modifying $K$ on an $\H^n$-null-set), $\rho_K$ is $C^\infty$ smooth, and hence (\ref{eq:RadRad}) and $I^2 K = c K$ hold pointwise. Moreover, either $\rho_K$ is identically zero (if $|K| = 0$) or else it is strictly positive, so the resulting modified $K$ is a Lipschitz star body, and the proof proceeds as usual. 
It is also worth noting that $IK$ can be defined for a bounded Borel set $K$,
but the equation $I^2 K = c K$ implies that $K$ is star-shaped.
It follows that the same characterization holds for bounded Borel sets as well.
\end{rem}

\begin{proof}[Proof of Corollary \ref{cor:intro-main}] 
If $K = B_n(R)$ then clearly $I K = R^{n-1} I B_n = R^{n-1} \omega_{n-1} B_n = R^{n-2} \omega_{n-1} K$. 
Conversely, if $I K = c K$ for some $c > 0$ then $I^2 K = c^{n-1} I K = c^n K$, and so $K$ must be a centered ellipsoid by Theorem \ref{thm:intro-main}. Writing $K = T (B_n)$ for some $T \in \GL_n$, we see by (\ref{eq:ITK}) that
\[
c T(B_n) = c K = I K = |\det T| (T^{-1})^*(I B_n) = |\det T| \omega_{n-1}  (T^{-1})^*(B_n) ,
\]
and hence $T^* T B_n = \frac{1}{c} |\det T| \omega_{n-1} B_n$. This implies that up to scaling, $T$ is orthogonal, and hence $K$ is a centered Euclidean ball. 
\end{proof}

\begin{proof}[Proof of Corollary \ref{cor:intro-SuK-equality}]
Let $K$ be a Lipschitz star body in $\R^n$, $n \geq 3$, and assume that $|I K| = |I(S_u K)|$ for all $u \in Q$ with $Q \subseteq \S^{n-1}$ of full-measure. By Theorem \ref{thm:Lip-multi-graphical} there exists a $\U \subseteq \S^{n-1}$ of full measure so that for all $u \in \U$, $K$ is $u$-multi-graphical, and so by Lemma \ref{lem:rectangles-increasing}, $[0,1] \ni t \mapsto |I(S_u^t K)|$ is non-decreasing. Consequently, for all $u \in \U \cap Q$, since $S_u^0 K = K$ and $S_u^1 K = S_u K$ by Corollary \ref{cor:Su01},\[
|I K| = |I(S_u^0 K)| \leq |I(S_u^1 K)| = |I(S_u K)| = |I K|,
\]
 we conclude that $[0,1] \ni t \mapsto |I(S_u^t K)|$ must be constant, and in particular $\dt{|I(S_u^t K)|} = 0$. Since $\U \cap Q$ is of full-measure in $\S^{n-1}$ and hence dense, we conclude as in the proof of Theorem \ref{thm:intro-main} that $K$ must be a centered ellipsoid by Theorems \ref{thm:mid-points} and \ref{thm:Soltan}. 
 
 Conversely, if $K$ is a centered ellipsoid, it is well known (e.g. \cite[Lemma 2]{BLM-SteinerSymmetrizations}) that $S_u K$ remains a centered ellipsoid (of the same volume) for all $u \in \S^{n-1}$, and so $|I K| = |I(S_u K)|$ by (\ref{eq:ITK}). 
\end{proof}

\section{Concluding remarks} \label{sec:conclude}

\subsection{Additional accessible results} 

The method we have employed in this work is rather general, and may be applied to characterize additional geometric equations. Let $\F(K)$ be a functional on the class of (Lipschitz) star bodies or convex bodies $K$ such that $[0,1] \ni t \mapsto \F(S_u^t K)$ is monotone under continuous Steiner symmetrization, and such that 
$\dt{\F(S^t_u K)} = 0$ for all (or a.e.) $u \in \S^{n-1}$ iff $K$ is an ellipsoid or a Euclidean ball (perhaps centered). Then any stationary point $K$ of $\F$ under admissible (i.e.~a.e.~equi-differentiable) perturbations of the radial function $\rho_K$ (for star bodies) or the support function $h_K = \sup_{x \in K} \scalar{\cdot,x}$ (for convex bodies) must be an ellipsoid or Euclidean ball, respectively. The stationary points for $\F$ are characterized by an Euler-Lagrange geometric equation, which is typically easy to compute, and so we obtain a method for generating and solving such geometric equations. 

\smallskip

Of course, to rigorously justify the above somewhat simplified sketch, one would need to handle some technicalities arising from employing continuous Steiner symmetrization $S_u^t K$. In this work, we have introduced this notion for Lipschitz star bodies $K$ and addressed the a.e.~equi-differentiability of $\rho_{S_u^t K}(\theta)$. The equi-differentiability of $h_{S_u^t K}(\theta)$ for convex bodies $K$ is actually much simpler, as this function is known to be convex in $t$ (see \cite[Lemma 2.1]{Saroglou-GeneralizedBCD} and the preceding comments). Note that convex bodies $K$ containing the origin in their interior are automatically Lipschitz star bodies, and so our results regarding $\rho_{S_u^t K}$ apply (see also \cite[Proposition 4.3]{Saroglou-GeneralizedBCD}). 

\smallskip
The literature already contains numerous functionals $\F$ for which $\F(K) \leq \F(S_u K)$ with equality for all $u \in \S^{n-1}$ iff $K$ is an ellipsoid or Euclidean ball (possibly centered). Often the arguments involve the use of continuous Steiner symmetrization, and it remains to inspect the proof and confirm that 
it is actually enough to have $\dt{\F(S^t_u K)} = 0$ for all $u \in \S^{n-1}$ to characterize ellipsoids or balls. 
Below is a partial list of geometric equations which may be solved using this approach --- we leave the details to the reader.  
From here on, $K$ denotes a convex body in $\R^n$. 

\begin{enumerate}
\item Fixed points for the centroid body of the polar projection body $\Pi^* K$: 
\begin{equation} \label{eq:Z1Pi}
\Gamma_1(\Pi^*K) = c K .
\end{equation}
The polar projection body $\Pi^* K$ is the polar body to the projection body $\Pi K$. It is the convex body whose gauge function is given by
\[
 \norm{\theta}_{\Pi^* K} = h_{\Pi K}(\theta) = |P_{\theta^{\perp}} K|_{n-1}.
 \]
 The centroid body $\Gamma_1(L)$ of $L$ is defined (up to our non-standard normalization) via $h_{\Gamma_1(L)}(\theta) = \int_L \abs{\scalar{\theta,x}} dx$. Of course, (\ref{eq:Z1Pi}) implies that the corresponding mixed surface area measures (see \cite{Lutwak-MixedProjectionInqs}) satisfy
\begin{equation} \label{eq:Z1Pi2}
S_{\Gamma_1(\Pi^* K),K,\ldots,K} = c \; S_{K,K,\ldots,K} . 
\end{equation}
It is easy to check that (\ref{eq:Z1Pi2}) is the Euler-Lagrange equation under perturbations of $h_K$ for the functional
\[
\F(K) =  |\Pi^* K| + c' |K| . 
\]
It is known that $\F(K) \leq \F(S_u K)$ \cite{LYZ-Lp-PettyProjection,LYZ-OrliczProjectionBodies}, and that equality occurs for all $u \in \S^{n-1}$ iff $K$ is an ellipsoid \cite[Theorem 1.4]{EMilmanYehudayoff-AffineQuermassintegrals}. In fact, it is shown in \cite[Theorem 3.10]{EMilmanYehudayoff-AffineQuermassintegrals} that $[0,1] \ni t \mapsto \F(S_u^t K)$ is non-decreasing. By inspecting and adapting the proof of \cite[Theorem 4.6]{EMilmanYehudayoff-AffineQuermassintegrals}, it is easy to check for a given $u \in \S^{n-1}$ that $\dt{\F(S_u^t K)} = 0$ iff $\F(S_u^t K)$ is constant for all $t \in [0,1]$, iff (by monotonicity) $\F(K) = \F(S_u K)$, and therefore $\dt{\F(S_u^t K)} = 0$ for all $u \in \S^{n-1}$ iff $K$ is an ellipsoid. 
 Consequently, it follows that (\ref{eq:Z1Pi2}) holds iff $K$ is an ellipsoid, and thus (\ref{eq:Z1Pi}) holds iff $K$ is an origin-symmetric ellipsoid (as the centroid body $\Gamma_1(\Pi^* K)$ is origin symmetric). This resolves a conjecture of Lutwak--Yang--Zhang from \cite[Section 7]{LYZ-Lp-PettyProjection} in the case that $p=1$; it is likely that the method can be extended to handle general $p \geq 1$, but we do not pursue this here. 
\item Fixed points for the iterated polar $L^p$-centroid body ($p \in [1,\infty)$):
\begin{equation} \label{eq:Zp*}
\Gamma_p^*(\Gamma_p^* K) = c K . 
\end{equation}
Here the polar $L^p$-centroid body $\Gamma_p^* K$ is the polar body to the $L^p$-centroid body $\Gamma_p K$, namely the convex body whose gauge function is given (up to our non-standard normalization) by
\[
 \norm{\theta}^p_{\Gamma_p^*(K)} = h^p_{\Gamma_p(K)}(\theta) = \int_K \abs{\scalar{\theta,x}}^p \, dx.
\]
 It is easy to check that (\ref{eq:Zp*}) is the Euler-Lagrange equation under perturbations of $\rho_K$ for the functional
\[
\F_p(K) = |\Gamma_p^* K| + \frac{n+p}{p c^{p}} |K| . 
\]
Since $\Gamma_p^*(L)$ is origin symmetric, it is enough to restrict to origin-symmetric $K = -K$ when considering solutions of (\ref{eq:Zp*}). 
It is known that $\F_p(K) \leq \F_p(S_u K)$ \cite[Lemma 3.2]{LutwakZhang-IntroduceLqCentroidBodies} with equality for all $u \in \S^{n-1}$ 
 iff $K$ is a centered ellipsoid (\cite[Proof of Theorem B]{LutwakZhang-IntroduceLqCentroidBodies} or \cite{CampiGronchi-VolumeProductInqs}). Moreover, defining $g_u(t-1) := 1 / |\Gamma_p^* S_u^t K|$, it was shown by Campi and Gronchi \cite[Theorem 2]{CampiGronchi-VolumeProductInqs} that $g_u : [-1,1] \rightarrow \R_+$ is a convex and even function (as $K = -K$). Consequently, for a given $u \in \S^{n-1}$, $\dt{\F(S_u^t K)} = 0$ iff $\dta{g_u(t)}{-1} = 0$ iff $g_u$ is constant on $[-1,1]$ iff $|\Gamma_p^* K| = |\Gamma_p^*(S_u K)|$, and so we conclude that $\dt{\F(S_u^t K)} = 0$ for all $u \in \S^{n-1}$ iff $K$ is a centered ellipsoid. 
Consequently, we confirm that (\ref{eq:Zp*}) holds iff $K$ is a centered ellipsoid (note that this is false for $p=\infty$, as $\Gamma_\infty^* K$ coincides with the polar body $K^*$ for all convex $K=-K$ and $(K^*)^* = K$). As in the proof of Corollary \ref{cor:intro-main}, it follows that
\[
\Gamma_p^* K = c K 
\]
iff $K$ is a centered Euclidean ball (and this also trivially holds for $p=\infty$). See \cite{Reuter-LocalFixedPoints} for some local fixed point results for various additional problems involving centroid bodies. 
\item Fixed points for the polar $L^p$-projection body of the $L^p$-centroid body ($p \in [1,\infty)$):
\begin{equation} \label{eq:polarLp}
\Pi_p^* (\Gamma_p K) = c K . 
\end{equation}
The $L^p$-centroid body $\Gamma_p K$ has already been defined above, and the polar $L^p$-projection body $\Pi_p^*(K)$ is the polar body to the $L^p$-projection body $\Pi_p(K)$, given (up to our non-standard normalization) by 
\[
\norm{\theta}^p_{\Pi_p^*(K)} = h^p_{\Pi_p(K)}(\theta) = \int_{\S^{n-1}} \abs{\scalar{\theta,\xi}}^p \, dS_p K(\xi).
\]
Here $S_p K$ denotes the $L^p$ surface area measure of the convex body $K$, defined as $S_p K = h_K^{1-p} S_K$, where $S_K = (\nu_{\partial K})_*(\H^{n-1}|_{\partial K})$ denotes the surface area measure of $K$ on $\S^{n-1}$ (and $\nu_{\partial K}$ is the unit outer normal to $\partial K$). These objects were introduced and studied by Lutwak in \cite{Lutwak-Firey-Sums,Lutwak-Firey-Sums-II}.
It is easy to show that (\ref{eq:polarLp}) is the Euler-Lagrange equation under perturbations of $\rho_K$ for the functional
\[
\F_p(K) = |\Gamma_p K| - c'_{p} |K|. 
\]
Since $\Pi_p^*(L)$ is origin symmetric, so is any solution $K$ to (\ref{eq:polarLp}). 
Defining $g_u(t-1) := |\Gamma_p S_u^t K|$, it was shown  by Campi and Gronchi \cite[Theorem 2.2]{CampiGronchi-LpBusemannPettyCentroid}  that $[-1,1] \ni t \mapsto g_u(t)$ is a convex function, which is trivially even whenever $K = -K$. Furthermore, \cite[Theorem 2.2]{CampiGronchi-LpBusemannPettyCentroid} shows that $|\Gamma_p K| = |\Gamma_p S_u K|$ for all $u \in \S^{n-1}$ iff $K$ is a centered ellipsoid. Consequently, $\dt{\F_{p}(S_u^t K)} = 0$ iff $\dta{g_u(t)}{-1} = 0$ iff $g_u$ is constant on $[-1,1]$ iff $|\Gamma_p K| = |\Gamma_p S_u K|$, and so we conclude that $\dt{\F_{p}(S_u^t K)} = 0$ for all $u \in \S^{n-1}$ iff $K$ is a centered ellipsoid. We thus confirm that (\ref{eq:polarLp}) holds iff $K$ is a centered ellipsoid. 
\item For completeness, we also mention the $L^p$-Minkowski equation ($p \geq -n$) for origin-symmetric convex bodies $K=-K$ and constant data:
\begin{equation} \label{eq:LpSK}
h_{K}^{1-p} S_K = c \, \H^{n-1}|_{\S^{n-1}} . 
\end{equation}
Recall that the left-hand side is precisely the $L^p$ surface area measure $S_p K$. 
It is known (see \cite{BCD-PowerOfGaussCurvatureFlow,IvakiEMilman-GeneralizedBCD,Lutwak-Firey-Sums,LYZ-LpMinkowskiProblem,Saroglou-GeneralizedBCD}) that (\ref{eq:LpSK}) holds iff $K$ is a centered ellipsoid (when $p=-n$) or a Euclidean ball (when $p > -n$, necessarily centered if $p \neq 1$ and necessarily with $c=c_n$ if $p=n$). 
It is easy to check that (\ref{eq:LpSK}) is the Euler-Lagrange equation under perturbations of $h_K$ for the functional
\[
\F_p(K) = \frac{1}{p} \int_{\S^{n-1}} h_K^{p}(\theta) \, d\theta + \frac{1}{c} |K| 
\]
(interpreted as $\F_0(K) = \int_{\S^{n-1}} \log h_K(\theta) \, d\theta + \frac{1}{c} |K|$ when $p=0$). Furthermore, when $K = -K$ is origin symmetric, it is known when $p \geq -n$ that $[0,1] \ni t \mapsto \F_p(S_u^t K)$ is non-decreasing: for $p=-n$, this was shown by Campi--Gronchi \cite[Theorem 1]{CampiGronchi-VolumeProductInqs}, and for $p > -n$ this follows from Saroglou's work \cite[Proposition 4.5]{Saroglou-GeneralizedBCD}. Moreover, it is known that $\F_{-n}(K) = \F_{-n}(S_u K)$ for all $u \in \S^{n-1}$ iff $K$ is a centered ellipsoid \cite{MeyerPajor-Santalo,SaintRaymond-Santalo}. Defining $g_u(t-1) := 1 / |(S_u^t K)^*|$, \cite[Theorem 1]{CampiGronchi-VolumeProductInqs} actually shows that $g_u : [-1,1] \rightarrow \R_+$ is a convex and even function, and so $\dt{\F_{-n}(S_u^t K)} = 0$ iff $\dta{g_u(t)}{-1} = 0$ iff $g_u$ is constant on $[-1,1]$ iff $|K^*| = |(S_u K)^*|$, and so we conclude that $\dt{\F_{-n}(S_u^t K)} = 0$ for all $u \in \S^{n-1}$ iff $K$ is a centered ellipsoid. 
When $p > -n$, \cite[Proposition 4.5 and Lemma 5.2]{Saroglou-GeneralizedBCD} imply that $\dt{\F_p(S_u^t K)} = 0$ for all $u \in \S^{n-1}$ iff $K$ is a centered Euclidean ball. These observations immediately recover the known results for origin-symmetric convex solutions to (\ref{eq:LpSK}) when $p \geq -n$. This variational proof is not new, and has been carried out (with all technical details) by Saroglou \cite[Proposition 5.1]{Saroglou-GeneralizedBCD} for $p > -n$; in fact, the hardest part of Saroglou's work is to treat the general case of (possibly not origin symmetric) convex bodies containing the origin in their interior. 
\item Of course, one may also combine several different functionals by adding or subtracting them, so that the overall monotonicity of $[0,1] \mapsto \F(S_u^t K)$ is preserved, yielding additional possibly interesting geometric equations. 
\end{enumerate}

\subsection{Inaccessible results} 

Before concluding, we mention a well-known dual problem to (\ref{eq:intro-I2K}), which remains inaccessible to our method. It was conjectured by Petty \cite{Petty-IsoperimetricProblems} that when $n \geq 3$, the quantity $\F(K) = |\Pi K| |K|^{1-n}$ is minimized over all convex bodies $K$ in $\R^n$ if and only if $K$ is an ellipsoid. Petty's projection conjecture is widely considered one of the major open problems in convex geometry; one reason this conjecture is apparently difficult is that $|\Pi K|$ may actually increase under Steiner symmetrization, as observed by Saroglou \cite{Saroglou-PettyNotMonotoneUnderSteiner}. It was observed by Schneider \cite[pp. 570-571]{Schneider-Book-2ndEd} 
that a necessary condition for $K$ to be a minimizer of $\F(K)$ is that
\[
\Pi^2 K = \frac{|\Pi K|}{|K|} K . 
\]
It is therefore very interesting to classify those convex bodies $K$ in $\R^n$ ($n \geq 3$) such that
\begin{equation} \label{eq:Schneider}
\Pi^2 K = c K ,
\end{equation}
which is clearly a dual problem to (\ref{eq:intro-I2K}). Contrary to (\ref{eq:intro-I2K}), for which centered ellipsoids are the only solutions, it is known that (\ref{eq:Schneider}) admits additional ones;
for example, even $K=B_\infty^n$ satisfies (\ref{eq:Schneider}).
The polytopes $K$ satisfying (\ref{eq:Schneider}) were completely classified by Weil \cite{Weil-Pi2KForPolytopes}. For additional partial results in these directions we refer to \cite{Ivaki-SecondMixedProjection,Ivaki-LocalUniquenessForPetty,SaroglouZvavitch-IterationsOfPi}. 

\medskip

Lastly, it is worthwhile mentioning Problem 5 of Busemann and Petty \cite{Busemann-Petty}, whose equivalent formulation (see \cite[Open Problem 12.6]{Lutwak-Selected}) asks whether the only (origin-symmetric) convex bodies satisfying $(I K)^* = c K$ for some $c>0$ are centered ellipsoids; this is known to be false in dimension $n=2$ but remains open for $n \geq 3$. We do not see how to extend our results in this direction. See \cite{ANRY-BusemannProblems5and8} for a solution when $n\geq 3$ and $K$ is close to the Euclidean ball in the Banach-Mazur distance.

\appendix

\section{Regularity of spherical Radon transform} \label{sec:appendix}

In this appendix, we establish the following {\it a priori} regularity for the solution to the equation $I^2 K = c K$ in the class of star-shaped bounded Borel sets in $\R^n$, $n \geq 3$. It will be clear from the proof that the same regularity holds equally for an equation of the form $I^\ell K = c K$ for any integer $\ell \geq 1$. Recall that $\Rad : L^2(\S^{n-1}) \rightarrow L^2(\S^{n-1})$ denotes the spherical Radon (or Funk) transform. 

\begin{thm} \label{thm:regularity}
Let $n \geq 3$, and let $f \in L^\infty(\S^{n-1})$ satisfy
\begin{equation} \label{eq:R-equation}
\Rad(\Rad(f^{n-1})^{n-1}) = c f ,
\end{equation}
for some $c \neq 0$. Then (possibly modifying $f$ on a null-set) $f \in C^\infty(\S^{n-1})$. In particular, if $f$ is {\it a priori} assumed continuous, then $f \in C^\infty(\S^{n-1})$.
Lastly, if $f$ is non-negative then either it is identically zero or else it is strictly positive. 
\end{thm}

Naturally, the proof relies on harmonic analysis, but also draws heavily from the theory of Sobolev spaces. 

\subsection{Harmonic analysis}

Let $n \geq 3$, and abbreviate $L^\infty = L^\infty(\S^{n-1})$ and $L^2 = L^2(\S^{n-1})$, noting that $L^\infty \subset L^2$. Given a real parameter $s \geq 0$, let $H^s = H^s(\S^{n-1})$ denote the (Bessel potential) fractional Sobolev space, consisting of all $f \in L^2$ such that $|D|^s f \in L^2$, or equivalently, all distributions $f$ 
such that $\di{D}^s f \in L^2$, where $|D| = (-\Delta)^{1/2}$, $\di{D} = (\Id - \Delta)^{1/2}$ and $\Delta$ is the spherical Laplacian. Set
\[
 \norm{f}_{H^s} := \norm{\di{D}^s f}_{L^2} \simeq \norm{f}_{L^2} + \norm{|D|^s f}_{L^2} . 
\]
Here and below, $a \simeq b$ indicates that $c_1 a \leq b \leq c_2 a$ for some constants $c_1,c_2 \in (0,\infty)$ that may depend on $n$ and $s$.

It is well known (e.g.~\cite{Groemer}) that if $Q_m$ is a spherical harmonic of degree $m \geq 0$ on $\S^{n-1}$, then
\[
-\Delta Q_m = m (m + n - 2) Q_m .
\]
Consequently, if
\[
f \sim \sum_{m=0}^\infty Q_m 
\]
denotes the (unique) decomposition of $f \in L^2$ into spherical harmonics $Q_m$ of degree $m$, then
\[
f \in H^s \;\; \Rightarrow \;\; \di{D}^s f \sim \sum_{m=0}^\infty  (1 + m (m + n - 2))^{\frac{s}{2}} Q_m .
\]
Therefore, by Parseval's identity, setting $\di{m} := \sqrt{1+|m|^2}$,
\begin{equation} \label{eq:Hs-norm}
f \in H^s \;\; \Leftrightarrow \;\; \norm{\di{D}^s f}_{L^2}^2 \simeq \sum_{m=0}^\infty \di{m}^{2s} \norm{Q_m}_2^2 < \infty ,
\end{equation}
and this can be used as a harmonic analytic definition of the space $H^s$.

\smallskip

The following lemma is well known.
\begin{lemma} \label{lem:R-regularity}
If $f \in H^{s}$ then $\Rad(f) \in H^{s + \frac{n}{2} - 1}$. 
\end{lemma}
\begin{proof}
This is explicitly proved in \cite[Lemma 4.3]{Strichartz}. 
Indeed, by \cite[Lemma 3.4.7]{Groemer} we have 
\[
\Rad(Q_m) = \nu_{n,m} Q_m 
\]
for an explicit constant $\nu_{n,m}$,  
which is easily seen (e.g. \cite[Proof of Lemma 3.4.8]{Groemer}) to satisfy
\[
\abs{\nu_{n,m}} \leq C_n m^{-\frac{n}{2} + 1} .
\]
The conclusion immediately follows from (\ref{eq:Hs-norm}). 
\end{proof}

\subsection{Algebraic structure of Sobolev spaces}

We will crucially need to use the following proposition, which already is more specialized and less known to non-experts (see \cite[Appendix]{KatoPonce-Commutators} for a proof for Euclidean space, \cite[Theorem 25]{CoulhonEtAl-SobolevAlgebras} for a proof for compact Riemannian manifolds, and \cite{BBR-AlgebraSobolevSpaces} for further extensions).
\begin{proposition} \label{prop:algebra}
For all $s \geq 0$, $H^s \cap L^\infty$ is an algebra: if $f, g \in H^s \cap L^\infty$ then $f g \in H^s \cap L^\infty$. 
\end{proposition}

For completeness and to better appreciate this non-obvious fact, we provide some context. First, note that Proposition \ref{prop:algebra} is completely false if we do not restrict to $L^\infty$, even for $H^0 = L^2$.  Second, for integer $k$ and $p \in [1,\infty]$, let $W^{k,p}$ denote the classical Sobolev space of functions whose first $k$ weak derivatives are in $L^p$. When $s$ is an integer, it is well known that $H^s$ coincides with the Sobolev space $W^{s,2}$. Using this and the Leibniz formula $\nabla( f g) = (\nabla f) g + (\nabla g) f$, it is very simple to show that $H^1 \cap L^\infty$ is an algebra. In order to extend this to higher \emph{integer} values of $s$, it is already necessary to use the classical Gagliardo--Nirenberg interpolation inequalities (see \cite[Propositions 31,32]{CoulhonEtAl-SobolevAlgebras} and the references therein for a proof in the Riemannian setting):
\begin{equation} \label{eq:GN}
\norm{f}_{W^{k,\frac{2s}{k}}} \leq C(k,s,n) \norm{f}_{W^{s,2}}^{\frac{k}{s}} \norm{f}_{L^\infty}^{1 - \frac{k}{s}} \;\; \forall k=1,\ldots,s-1 . 
\end{equation}
For example, since $\Delta (f g) = \Delta(f) g + 2 \scalar{\nabla f, \nabla g} + f \Delta(g)$, in order to show that this is in $L^2$ by invoking Cauchy-Schwarz on the $\scalar{\nabla f , \nabla g}$ term, one needs to establish that if $f \in H^2 \cap L^\infty$ then $\abs{\nabla f} \in L^4$, i.e.~that $f \in W^{1,4}$, which is precisely guaranteed by (\ref{eq:GN}). This is already enough for establishing Proposition \ref{prop:algebra} and hence Theorem \ref{thm:regularity} when $n \geq 4$, since in that case $n/2 - 1 \geq 1$ and so Lemma \ref{lem:R-regularity} guarantees that the Radon transform adds at least one derivative of regularity, allowing us to only work with integer values of $s$. While (\ref{eq:GN}) remains valid in our setting also for fractional values of $s,k$ (perhaps with some exceptional limiting cases, depending on how one interprets $W^{s,p}$ for fractional $s$ --- see \cite{BrezisMironescu-TheFullStory,CoulhonEtAl-SobolevAlgebras}),
it is no longer clear how to invoke the Leibniz formula for fractional derivatives. Consequently, a different approach is required to handle fractional values of $s$, such as the half-integer values which will appear in the proof of Theorem \ref{thm:regularity} in the case $n=3$ (for which $n/2  - 1 = 1/2$). 
In \cite[Appendix]{KatoPonce-Commutators}, Kato and Ponce established Proposition \ref{prop:algebra} in the Euclidean setting by essentially proving the following ``fractional Leibniz" inequality (for more general $H^{s,p}$ spaces) on $\R^d$.
\begin{thm}[Kato--Ponce Inequality] \label{thm:Kato-Ponce}
For all $s \geq 0$ and smooth functions $f, g \in C^\infty(\S^{n-1})$,
\[
\norm{f g}_{H^s} \leq C(s,n) \brac{ \norm{f}_{H^s} \norm{g}_{L^\infty} + \norm{g}_{H^s} \norm{f}_{L^\infty} } .
\]
\end{thm}
As already explained, for integer values of $s$ this follows easily from the Gagliardo--Nirenberg inequalities, but the general case requires the theory of bilinear multipliers. See \cite[Theorem 7.6.1]{Grafakaos-ModernFourierAnalysis3rdEd} and the references therein for generalizations, and \cite[Theorem 25]{CoulhonEtAl-SobolevAlgebras} for an extension to the Riemannian setting, which in particular applies to compact Riemannian manifolds such as $\S^{n-1}$. 
As smooth functions are dense in $H^s$, Theorem \ref{thm:Kato-Ponce} immediately implies Proposition \ref{prop:algebra}.

\medskip

Combining all of the above, we immediately obtain the following proposition. 
\begin{proposition} \label{prop:Rfk}
If $f \in H^{s} \cap L^\infty$ then $\Rad(f^k) \in H^{s + \frac{n}{2} - 1} \cap L^\infty$ for all integers $k \geq 1$. 
\end{proposition}

\subsection{Concluding the proof}

\begin{proof}[Proof of Theorem \ref{thm:regularity}]
Applying Proposition \ref{prop:Rfk} twice, it follows for all $s\geq 0$ that if $f \in H^s \cap L^\infty$ satisfies (\ref{eq:R-equation}), then in fact $f \in H^{s + n -2} \cap L^\infty$. Applying this repeatedly starting from $s=0$ (as $L^\infty \subset L^2$), we deduce that $f \in H^{k (n-2)} \cap L^\infty$ for all integer $k \geq 1$. It follows by a standard application of the Sobolev--Morrey embedding theorem \cite[Theorem 6.3]{HebeyRobert-SobolevSpacesOnManifolds} that, up to modifying $f$ on a null-set,  $f$ is $C^\infty$-smooth, as asserted. 

If $f \in C^\infty(\S^{n-1},\R_+)$ and $f(\theta_0) = 0$, then denoting $g = \Rad(f^{n-1}) \in C^\infty(\S^{n-1},\R_+)$, we are given that $\Rad(g^{n-1})(\theta_0) = c f(\theta_0) = 0$, which clearly implies that $g$ vanishes on $\S^{n-1} \cap \theta_0^{\perp}$. But this in turn implies that $f$ vanishes on $\S^{n-1} \cap u^{\perp}$ for all $u \in \S^{n-1} \cap \theta_0^{\perp}$, meaning that $f$ vanishes on the entire $\S^{n-1}$. Consequently, if $f$ is not identically zero, then $f$ must be strictly positive. 
\end{proof}

\bibliographystyle{plain}

\def\cprime{$'$} \def\textasciitilde{$\sim$}

\end{document}